\documentclass[12pt,leqno]{article}

\usepackage{amsfonts,amsthm,amssymb,amsmath,graphicx}
\usepackage{enumerate}

\usepackage{color}

\newcommand{\fg}{\mathfrak {g}}

\newcommand{\fh}{\mathfrak {h}}
\newcommand{\fn}{\mathfrak {n}}
\newcommand{\ft}{\mathfrak {t}}
\newcommand{\fa}{\mathfrak {a}}
\newcommand{\fk}{\mathfrak {k}}

\newcommand{\fm}{\mathfrak {m}}

\newcommand{\fp}{\mathfrak {p}}

\newcommand{\fs}{\mathfrak {s}}

\newcommand{\fu}{\mathfrak {u}}
\newcommand{\fz}{\mathfrak {z}}

\newcommand{\fgc}{\mathfrak {g}_{\mathbb {C}}}

\newcommand{\si}{\sigma}

\newcommand{\sk}{\bigskip}

\newcommand{\lra}{\longrightarrow}

\newcommand{\caH}{{\cal H}}

\newcommand{\pa}{\partial}

\newcommand{\be}{\beta}
\newcommand{\de}{\delta}

\newcommand{\la}{\lambda}

\newcommand{\ra}{\rightarrow}

\newcommand{\im}{{\rm Im}}

\newcommand{\tr}{{\rm trace}\,}

\newcommand{\ch}{\mathcal{H}}

\newcommand{\ad}{{\rm ad}}

\newtheorem{theorem}[equation]{Theorem}

\newtheorem{lemma}[equation]{Lemma}

\newtheorem{corollary}[equation]{Corollary}

\newtheorem{proposition}[equation]{Proposition}

\newtheorem{remark}[equation]{Remark}

\def\sideremark#1{\ifvmode\leavevmode\fi\vadjust{\vbox to0pt{\vss
 \hbox to 0pt{\hskip\hsize\hskip1em
\vbox{\hsize2cm\tiny\raggedright\pretolerance10000 
 \noindent #1\hfill}\hss}\vbox to8pt{\vfil}\vss}}} 

\newcommand{\bz}{\mathbb {Z}}
\newcommand{\bc}{\mathbb {C}}
\newcommand{\bbC}{\mathbb {C}}
\newcommand{\bbR}{\mathbb {R}}

\newcommand{\br}{\mathbb {R}}
\newcommand{\bd}{\mathbb {D}}
\newcommand{\bbD}{\mathbb {D}}

\newcommand{\wh}{\widehat}
\newcommand{\ind}{\rm Ind}
\newcommand{\Ad}{\mathrm{Ad}}
\newcommand{\beq}{\begin{equation}}
\newcommand{\eeq}{\end{equation}}
\newcommand{\Dt}{\widetilde D}
\newcommand{\we}{\wedge}
\newcommand{\om}{\omega}
\newcommand{\Om}{\Omega}
\newcommand{\bl}{\mathbb {L}}
\newcommand{\fgg}{\mathfrak {g}}
\newcommand{\fhh}{\mathfrak {h}}

\newcommand{\ftc}{\mathfrak {t}_\bc}

\newcommand{\fmc}{\mathfrak {m}_\bc}

\newcommand{\fsc}{\mathfrak {s}_\bc}
\newcommand{\fkc}{\mathfrak {k}_\bc}
\newcommand{\blx}{\bl}

\newcommand{\caHH}{\caH^2}
\newcommand{\wt}{\widetilde}
\newcommand{\hi}{\frac{i}{2}}
\newcommand{\hH}{H^0}
\newcommand{\tT}{T^0}

\begin{document}

\begin{center}
{\Large \bf Partial Dirac Cohomology and Tempered Representations}
\end{center}

\sk

\begin{center}
{\large Meng-Kiat Chuah\footnote{Department of Mathematics, National Tsing 
Hua University, Hsinchu, Taiwan; 
research supported in part by the Ministry of Science and Technology of Taiwan;
e--mail {\tt chuah@math.nthu.edu.tw}.}, 
Jing-Song Huang\footnote{School of Science and Engineering, Chinese 
University of Hong Kong, Shenzhen, China; research supported in part by the 
Research Grant Council of Hong Kong SAR, China; e–mail
{\tt huangjingsong@cuhk.edu.cn}.} and
Joseph A. Wolf\,\footnote{Department of Mathematics, University of California,
Berkeley CA 94720-3840, USA; 
research partially supported by a Simons Foundation grant;
e--mail {\tt jawolf@math.berkeley.edu}.\newline
\phantom{Gra} {\it Key words:} real reductive group, tempered representation, 
Plancherel decomposition, Dirac cohomology, geometric quantization \newline
\phantom{Gra} {\it AMS subject class 2010:} 22E45, 22E46, 22F30, 53D50}}
\end{center}

\begin{center} 11 February 2022 \end{center}

\begin{abstract}
The tempered representations of a real reductive Lie group $G$
are naturally partitioned into series associated with
conjugacy classes of Cartan subgroups $H$ of $G$.
We define partial Dirac cohomology, apply it for geometric construction
of various models of these $H$--series representations, and show how this 
construction fits
into the framework of geometric quantization and symplectic reduction.
\end{abstract}

\section{Introduction}\label{sec1}
\setcounter{equation}{0}

Let $G$ be a real reductive Lie group.
Then every conjugacy class of Cartan subgroups $H \subset G$ defines a series
of unitary representations that enters into the Plancherel formula in an
essential manner.  For simplicity of exposition we assume
that $G$ is of Harish--Chandra class, and later we will
relax that condition to include groups such as the universal
covering group of $SU(p,q)$.  Then the Cartan subgroups $H$ defines cuspidal 
parabolic subgroups $MAN \subset G$, meaning that $H = TA$ where $T$ is a 
compact Cartan subgroup of $M$ and $A$ is split over $\bbR$, so $M$ has 
discrete series representations.
The $H$--series of unitary $G$-representations consists of the
unitarily induced ${{\ind}}_{MAN}^G (\eta \otimes e^{i\sigma} \otimes 1)$
where $\eta \in \wh{M}_{disc}$ (unitary equivalence classes of discrete
series representations of $M$), and $\sigma \in \fa^*$ (so that
$e^{i\sigma}$ is a unitary character on $A$).  Dirac cohomology was studied
in \cite{HP2002}, settling Vogan's Conjecture.  Dirac cohomology reveals
the infinitesimal characters of Harish-Chandra modules attached to the
$H$--series for $H$ as compact as possible, i.e. for the situation where
$T$ is a Cartan subgroup of a maximal compact subgroup $K \subset G$.
That is usually called the {\em fundamental series} of $G$, and it is the
{\em discrete series} just when $H = T$.

Here we study the partial Dirac cohomology with respect to $H$ in general,
and express it as a direct integral of $H$-series representations.
We say ``partial'' because we are using Dirac cohomology of representations
of the subgroup $M$ for the purpose of constructing representations of $G$.
As an application, we incorporate partial Dirac cohomology into geometric
quantization \cite{K1970} to construct unitary $G$-representations, and we
show that the occurrences of $H$-series is controlled by the image of the
moment map.  


There are two approaches.  One is to work directly with the class of general
real reductive Lie groups introduced and studied in \cite{W1974a} and
updated in \cite{W2018}.  The other is to first work with groups of
Harish--Chandra class, and then show how the results extend
{\em mutatis mutandis} to general real reductive Lie groups.  For clarity
of exposition we do the second of these.  Thus, in Sections \ref{sec2}
through \ref{sec6}, $G$ will be of Harish--Chandra class.  In
Sections \ref{sec7} through \ref{sec9}, $G$ will be a general real reductive 
Lie group.  

Our main results are stated in Section \ref{newsec2} as Theorems \ref{thm1}, 
\ref{thm2} and \ref{thm3} and as their
extensions to the class of general real reductive Lie groups.

The sections of this article are arranged as follows.

In Section \ref{newsec2} we develop the background material, give
a detailed introduction, and formulate our main results.

In Section \ref{sec2} we study the Plancherel decomposition
of $L^2(G/N)$ and use it to prove Proposition \ref{sat}.
In Section \ref{sec3} we use Proposition \ref{sat} to prove
Theorem \ref{thm1}.  

In Section \ref{sec4} we study several $L^2$-function spaces
and prove Proposition \ref{keyy}.
In Section \ref{sec5} we use Proposition \ref{keyy} to prove
Theorem \ref{thm2}.

In Section \ref{sec6} we perform symplectic reduction and prove
Theorem \ref{thm3}.

In Section \ref{sec7} we formulate the extension of the theorems
from groups of Harish-Chandra class to general real reductive Lie groups.

In Section \ref{sec8} we develop a method of reducing questions
of general real reductive Lie groups $G$ to the case where $G^\dagger$
has compact center.
In Section \ref{sec9} we use the tools of Section \ref{sec8}
to reduce the proofs of the theorems of Section \ref{sec7},
to the arguments used for groups of Harish--Chandra class.

\section{Statement of Results}\label{newsec2}
\setcounter{equation}{0}

As usual we use the lower case Gothic letters for
the real Lie algebras,
and add the subscript $\bc$ for their complexifications.
So for example $\fgg$ is the Lie algebra of $G$, and $\fgc$
is its complexification.

If $L$ is a Lie group then $\widehat{L}$ denotes its unitary dual.  If
$\pi \in \widehat{L}$ then $\mathcal{H}_\pi$ denotes its representation
space.  If $Z$ is a central subgroup of $L$ and $\zeta \in \widehat{Z}$
is a unitary character then
$\widehat{L}_\zeta = \{\pi \in \widehat{L} \mid \pi|_Z \text{ is a multiple
of } \zeta\}$.

The {\em Harish--Chandra class} of Lie groups consists of
all real reductive Lie groups $G$ such that
\begin{equation} \label{HC-class}
\begin{aligned}
  &\text{the identity component $G^0$ has finite index in $G$,}\\
  &\text{the derived group $G' = [G^0,G^0]$ of $G^0$ has finite center, and}\\
  &\text{if $g \in G$ then $\Ad(g)$ is an inner automorphism of $\fgc$.}
\end{aligned}
\end{equation}
The Lie algebra of $G$ is $\fg = \fg' \oplus \fz$ where $\fg' = [\fg,\fg]$
is the semisimple part and $\fz$ is the center.  The second condition in
(\ref{HC-class}) is that the semisimple part of $G^0$ (the connected 
subgroup of $G^0$ with Lie algebra $\fg' = [\fg,\fg]$) has finite center.

Let $H$ be a Cartan subgroup of $G$ and $\theta$ a Cartan involution that
stabilizes $H$.  See \cite{W1974a} (or \cite{W2018}) for existence of
Cartan involutions of groups of Harish--Chandra class.  Then the fixed
point set $K = G^\theta$ is a maximal
compact subgroup of $G$.  The Lie algebra $\fg$ of $G$ is the sum
$\fg = \fk + \fp$ of $\pm 1$ eigenspaces of $\theta$. We have decompositions
$\fh = \ft + \fa$ of the Lie algebra, $\ft = \fh \cap \fk$ and
$\fa = \fh \cap \fp$, and $H = T \times A$ on the group
level where $T = H \cap K$ and $A = \exp(\fa)$.  The centralizer
$Z_G(A) = M \times A$ where $T$ is a compact Cartan subgroup of $M$.
Further, $M$ is of Harish--Chandra class.  Choose a positive
$\fa$--root system $\Sigma_{\fg,\fa}^+$ on $\fg$.  Let $\fn$ denote the sum of
the positive $\fa$--root spaces and let $N = \exp(\fn)$.  Then
the parabolic subgroup $P = MAN$ of $G$ is {\em a cuspidal 
parabolic subgroup associated 
to} $H$.  Cuspidal parabolics are characterized by the fact that $M$
has a compact Cartan subgroup, in other words that $M$ has discrete series
representations.

When $G$ has discrete series representations, in other words when $G$
has a compact Cartan subgroup $T$, we avoid dealing with projective
representations as follows.  Replace $G$ by a double covering if necessary
so that $e^{\rho_\fg}$ is a well defined unitary character on $T^0$,
where $\rho_\fg := \tfrac{1}{2}\sum_{\alpha \in \Sigma_{\fgc,\ftc}^+}\,\alpha$\,.
Then we say that $G$ is {\em acceptable}.
Acceptability is independent of the choice of compact Cartan
subgroup $T$ because any two are $\Ad(G^0)$--conjugate, and independent of
the choice of positive $\ft$--root system because any two such $\rho_\fg$
differ by a linear combination of $\ft$--roots.  See \cite{W1974a} for the
construction of acceptable double covers.

We now assume that all our groups $M$ (including $G$ when it has a compact
Cartan subgroup) are acceptable.  Thus for each $M$,
$e^{\rho_\fm}$ is a well defined unitary character on $T^0$
where $\rho_\fm$ is half the sum over a positive $\ft_\mathbb{C}$--root
system of $\fmc$.

Since $M$ has the compact Cartan subgroup $T$ it has discrete series
representations $\eta_{\chi,\lambda}$ parameterized as follows.
$\lambda \in i\ft^*$ belongs to the weight lattice for finite dimensional
representations of $M$.  In other words it satisfies the integrality
condition that $e^\lambda$ is a well defined character on $T^0$.
Further, $\lambda$ is regular in the sense that
$\langle \lambda, \alpha \rangle
\ne 0$ for every root $\alpha \in \Sigma_{\fmc,\ftc}$.  Let
$\Delta_{M^0}$ denote the Weyl denominator,
$\Delta_{M^0} = \prod_{\alpha \in \Sigma_{\fmc,\ftc}^+}
(e^{\alpha/2} - e^{-\alpha/2})$.  The the corresponding discrete series
representation for $M^0$ and its distribution character are
\begin{equation}\label{eta0}
\eta^0_\lambda \in \wh{M^0}_{disc} \text{ with } 
	\Theta_{\eta^0_\lambda} 
	= \pm \tfrac{1}{\Delta_{M^0}} {\sum}_{w \in W_{M^0}} 
	\det(w)e^{w(\lambda)}
\text{ on } (M^0)''\cap T^0
\end{equation}
where $(M^0)''$ is the $M$--regular set in $M^0$.
The representation $\eta^0_\lambda$ 
has infinitesimal character with Harish--Chandra parameter
$\lambda$ and central character $e^{\lambda - \rho_\fm}|_{Z_{M^0}}$.
This last is why we require $M$ to be acceptable.
See Harish--Chandra \cite[Theorems 13 and 16]{HC1966}, or
\cite[Theorem 3.4.7]{W1974a} (\cite[Theorem 3.4.4]{W2018})
for an extension to general real reductive Lie groups.

Define $\varpi(\lambda):=
\prod_{\alpha \in \Sigma_{\fmc,\ftc}^+} \langle \alpha ,\lambda\rangle$.  Then
$\eta^0_\lambda$ has {\rm formal degree} $\deg(\eta^0_\lambda) =
|\varpi(\lambda)|$.  This replaces the degree that appears in
the Peter--Weyl Theorem for compact groups.

Suppose that $\chi \in \widehat{Z_M(M^0)}$ agrees with
$\eta^0_\lambda$ on the center $Z_{M^0}$, so
$\chi \in \widehat{Z_M(M^0)}_\zeta$ and $\eta^0_\lambda \in
\wh{M^0}_\zeta$ for the same $\zeta = e^{\lambda - \rho_\fm}|_{Z_{M^0}}
\in \widehat{Z_{M^0}}$.  Define $M^\dagger := Z_M(M^0) M^0$\,.  Then
\begin{equation}\label{etadag}
\eta^\dagger_{\chi,\lambda} :=
\chi \otimes \eta^0_\lambda \in \wh{M^\dagger}_{disc} \text{ with }
\Theta_{\eta^\dagger_{\chi,\lambda}}(zm) = \tr \chi(z)
	\Theta_{\eta^0_\lambda}(m)
\end{equation}
for $z \in Z_M(M^0)$ and $m \in M^0$\,.	
It has infinitesimal character of Harish--Chandra parameter
$\lambda$, central character $e^{\lambda - \rho_\fm}|_{Z_{M^0}}$\,,
and formal degree $\deg(\eta^\dagger_{\chi,\lambda}) =
\deg(\chi)|\varpi(\lambda)|$.

The discrete series representations of $M$ are the
$\eta_{\chi,\lambda} := {\ind}_{M^\dagger}^M(\eta^\dagger_{\chi,\lambda})
\in \wh{M}_{disc}$\,.  Their
distribution characters are supported in $M^\dagger$, where 
\begin{equation}\label{def-ds-M}
\eta_{\chi,\lambda} := {\ind}_{M^\dagger}^M(\eta^\dagger_{\chi,\lambda})
\text{ with }
\Theta_{\eta_{\chi,\lambda}}(zm) = {\sum}_{xM^\dagger \in M/M^\dagger}
	\tr \chi(x^{-1}zx) \Theta_{\eta^0_\lambda}(x^{-1}mx)
\end{equation}
for $z \in Z_M(M^0)$ and $m \in M^0$\,.  The representation 
$\eta_{\chi,\lambda}$ has infinitesimal character of Harish--Chandra parameter
$\lambda$ and formal degree $\deg(\eta_{\chi,\lambda}) =
|M/M^\dagger|\deg(\chi)|\varpi(\lambda)|$.
Every discrete series representation of $M$ is one of the $\eta_{\chi,\lambda}$.
Further $\eta_{\chi,\lambda} \simeq \eta_{\chi',\lambda'}$
just when they have equivalent restrictions to $M^\dagger$\,.

Let $\sigma \in \fa^*$.  Then $e^{i\sigma}$
is a unitary character on $A$, and
\begin{equation}\label{def-ds-G}
\pi_{\chi,\lambda,\sigma} :=
{\ind}_{MAN}^G(\eta_{\chi,\lambda} \otimes e^{i\sigma} \otimes 1)
= {\ind}_{M^\dagger AN}^G(\eta^\dagger_{\chi,\lambda}
	\otimes e^{i\sigma} \otimes 1)
\end{equation}
is a unitary representation of $G$.  These representations form
the $H$--{\sl series}.  The $H$--series depends only on the $G$--conjugacy class
of $H$.  The discrete series is the case where $H$
is compact (i.e. $H = T$), and the principal series is the case
where $H$ is as noncompact as possible.  Let $Car(G)$ denote the set
of $G$--conjugacy classes of Cartan subgroups.  It is finite.
Each $\{H\} \in Car(G)$ contributes a term in the Plancherel and
Fourier Inversion formulae of $G$, and those formulae are the
sums of those terms.  See Harish--Chandra (\cite{HC1975}, \cite{HC1976a},
\cite{HC1976b}).

The various $H$--series representations $\pi_{\chi,\lambda,\sigma}$ are
called the {\em standard tempered representations} of $G$.  
Plancherel-almost-all of them are irreducible.  For example if $\sigma$
is regular relative to the $\fa$--roots of $\fg$ then 
$\pi_{\chi,\lambda,\sigma}$ is irreducible.  In any case, every standard 
tempered representation of $G$ is a finite sum of irreducibles, and every
tempered representation of $G$ is a summand of a standard 
tempered representation.

Since $\Ad_G(N)$ is a unipotent group of linear transformations of $\fgg$,
it preserves Haar measure $dg$ and thus defines a $G$--invariant
measure $d(gN)$ on $G/N$.  Consider
$$
L^2(G/N) =
\left \{f : G/N \lra \bc \left | \int_{G/N} |f(g)|^2 \, d(gN) < \infty
	\right . \right \}.
$$
Since $M$ normalizes $N$, it acts on $G/N$ from the right,
and it preserves $d(gN)$.
So the action of $G \times M$ on $G/N$ leads to a unitary representation
of $G \times M$ on $L^2(G/N)$.
We now show how the Plancherel decomposition of
$L^2(M)$ as an $M\times M$ module induces a
Plancherel decomposition of $L^2(G/N)$ as $G \times M$ module.

Since $\fgg$ is real reductive, it has a non-degenerate
symmetric $\Ad(G)$--invariant bilinear form $b$.  We
choose $b$ to be the Killing form on $[\fgg,\fgg]$, negative definite
on $\fk$ and positive definite on $\fp$.

Let $\fmc = \ftc \oplus \fsc$\,, orthogonal direct sum.
Let $S$ be the spin module for the Clifford algebra $C(\fsc)$.
Let
\begin{equation}\label{mdirac}
D_M = D_{(\fmc, \fmc \cap \fkc)} + i D_{diag ; (\fmc \cap \fkc, \ftc)} 
\end{equation}
be the modified Dirac operator as defined in \cite[(1.1)]{CH2016}.
It induces a densely defined symmetric operator,
whose closure is a symmetric operator
\begin{equation} \label{dira0}
 \bd_M : L^2(G/N) \otimes S \lra L^2(G/N) \otimes S .
\end{equation}

Since $M$ is of Harish--Chandra class, $T = Z_M(M^0)T^0$, so every
irreducible unitary representation of $H = T \times A$ has form
\begin{equation}\label{mmm}
\chi_{\lambda,\sigma}:=\chi \otimes e^\lambda \otimes e^{i\sigma}\,\,,
	\lambda \in i\ft^*, \sigma \in \fa^*, \chi \in \wh{Z_M(M^0)}
	 \text{ agrees with $e^\lambda$ on $Z_{M^0}\,.$}
\end{equation}
If $G$ is a linear group then $H$ is commutative
and $\wh{H}$ is the set of all unitary characters on $H$.  In any case these
representations are finite dimensional.

The intersection $K \cap M$ is a maximal compact subgroup of $M$.
Denote its Weyl group by $W_{K \cap M}$.
Then $W_{K \cap M}=W(M,T)=N_M(T)/Z_M(T)\cong N_{K \cap M}(T)/Z_{K \cap M}(T)$.
In general if $\psi$ is a unitary representation we write $\ch_\psi$ for its
representation space.  If $\ch_1$ and $\ch_2$ are separable Hilbert spaces we
write $\ch_1 \wh{\otimes} \ch_2$ for their projective tensor product.  If one
or both of the $\ch_i$ is finite dimensional then it is the ordinary
tensor product.  Also note that if $\chi \in \wh{Z_M(M^0})$ then
$\bar\chi$ is the same as the dual $\chi^*$.  Further the action of
$W_{K \cap M}$ on $\wh{Z_M(M^0})$ is $w: \chi \mapsto \chi\cdot w^{-1}$.

\begin{theorem} \label{thm1}
Let $G$ be a real reductive Lie group of Harish--Chandra class.
Recall the notation {\rm (\ref{def-ds-M})} and {\rm (\ref{def-ds-G})} for the
representations $\eta_{\chi,\lambda}$ of $M$ and
$\pi_{\chi,\lambda,\sigma}$ of $G$.  Then $\bbD_M$ has kernel
$$
{\sum}_{\eta_{\chi, \lambda} \in\widehat M_{disc}}
	\left (\int_{\sigma\in\fa^*} \ch_{\pi_{\chi,\lambda,\sigma}}
		\, d\sigma \right )
\otimes \left ({\sum}_{w\in W_{K \cap M}}
        \ch_{(\bar\chi\cdot w^{-1} \otimes e^{-w\lambda})} \right ),
$$
and the natural action of $G \times T$
on that kernel is the direct integral
$$
{\sum}_{\eta_{\chi, \lambda} \in\widehat M_{disc}}
	\left (\int_{\sigma\in\fa^*} \pi_{\chi,\lambda,\sigma}
		\, d\sigma \right )
\otimes \left ({\sum}_{w\in W_{K \cap M}}
       (\bar\chi\cdot w^{-1} \otimes e^{-w\lambda}) \right ).
$$
\end{theorem}

We shall incorporate partial Dirac cohomology into the setting of
symplectic geometry.
In particular, we consider geometric quantization on
\[ X = G \times \fhh ,\]
where a symplectic manifold $(X,\om)$ with symmetry leads to
a unitary representation $\pi_{(X,\om)}$ \cite{K1970}.  For our case,
compare \cite{W1974b}.

There is a natural action of $G \times G$ on $X$ given by the
left and right actions of $G$.
In Theorem \ref{assam}, we recall a systematic construction of some
$(G \times H)$--invariant symplectic forms $\om$ on $X$.
There is a complex line bundle $\blx$ over $X$ whose Chern class is the
cohomology class of $\om$, equipped with a connection whose curvature is $\om$,
and also equipped with an invariant Hermitian structure.
One may use the Hermitian structure and the invariant measure on $X$
to construct a unitary representation of $G \times T$ on the Hilbert space
of $L^2$-sections of $\bl$.
That space of $L^2$-sections is too big to yield an irreducible representation,
so extra conditions are imposed, a process called polarization that cuts the
number of variables in half.  For example holomorphic sections of $\blx$
would be considered if $X$ is complex.  Here $X$ is not complex, but it has
a partial complex structure in the sense that the fibers of
\[ \pi: X \lra G/H \]
are complex (see (\ref{dapat})).
A section of $\blx$ is called $\pi$--{\it holomorphic} or {\it partially
holomorphic} if it is holomorphic on each fiber of $\pi$.
In \cite{C2014}, using the partial Dolbeault cohomology introduced by one of us
(\cite{W1973}, \cite{W1974a}, or see \cite{W2018}), this idea is applied to
construct some of the unitary representations
$\pi_{\chi,\lambda,\sigma}\otimes (\bar\chi\cdot w^{-1} \otimes e^{-w\lambda})$
on higher degree cohomology.
We shall apply the Dirac operator as in \cite{W1974c} to simplify
the polarization process, constructing $\caH_{(X,\om)}$
without involving higher degree Dolbeault cohomology.

Consider the inclusion
\[ \imath : G \lra X \;,\,
\imath(g) = (g,0) .\]
Then $\imath^* \bl$ is a line bundle on $G$.
Let $L^2(\bl)^N$ (resp. $L^2(\imath^* \bl)^N$)
denote the right $N$-invariant sections $f$
of $\bl$ (resp. $\imath^* \bl$) such that $(f,f)$ is integrable over
$\int_{G/N \times \fh} (\cdot) \, d(gN) \, dx$ 
(resp. $\int_{G/N} (\cdot) \, d(gN)$).
As in (\ref{dira0}), we have a Dirac operator
$\bbD_\bl$ on $L^2(\imath^* \bl)^N \otimes S$.
Define
\[
\begin{aligned}
&\caH_{(X,\om)} = \{ f \in L^2(\bl)^N \otimes S
\mid \imath^* f \in {\rm Ker}\, \bbD_\bl
\mbox{ and } f \mbox{ is $\pi$--holomorphic}\}, \\
&\pi_{(X,\om)}: \text{ natural representation of } G \text{ on }
\caH_{(X,\om)\,.}
\end{aligned}
\]
The action of $G \times H$ on the symplectic manifold $(X,\om)$
is Hamiltonian.  In particular the right $H$ action has a
canonical moment map \cite{GS1984}
\[ \Phi : X \lra \fh^* .\]
We modify it by $i$ so that the image set satisfies
$\im(i \Phi) \subset i \fh^*$.

By Theorem \ref{assam} and (\ref{cano}), each
$\la + i\sigma \in \im(i \Phi)$ is $MA$--regular,
so it determines the family of all representations
$\eta_{\chi,\lambda}\otimes e^{i\sigma} \in \wh{M}_{disc} \otimes \wh{A}$
as described in (\ref{def-ds-M}).
Thus it determines the family of all $H$--series representations
$\pi_{\chi,\lambda,\sigma}$ of $G$, as in (\ref{def-ds-G}).
Recall $\rho_\fa$, half the sum of positive $\fa$-roots
relative to $N$. Let $d \la$ be the counting measure.

\begin{theorem} \label{thm2}
Let $G$ be a real reductive Lie group of Harish--Chandra class.
As a unitary representation space for $G \times T$,
\[
\caH_{(X, \om)} \cong 
\int_{\tiny \begin{array}{l} \la + i\sigma \in \im(i \Phi)\\
      \exp(\la + i\sigma) \in \wh{H^0} \end{array}}
{\sum}_{\tiny \chi \in \widehat{Z_M(M^0)}_\lambda}
  \left ( {\ind}_P^G(\caH_{\bar\chi \otimes \eta^0_{-\la}}
       \otimes e^{-i\sigma +\rho_\fa} \otimes 1)
       \otimes \ch_{(\chi\otimes e^\lambda)} \right ) \, d \la \, d \si
\]
where $\widehat{Z_M(M^0)}_\lambda$ denotes the elements of $\widehat{Z_M(M^0)}$
that agree with $e^\lambda$ on $Z_{M^0}$\,.
\end{theorem}

In Corollary \ref{spc}, we use the case $A=1$ of Theorem \ref{thm2}
to correct an error in \cite[Theorem B]{CH2016}.

According to Gelfand, a {\em model} of a compact Lie group
is a multiplicity-free unitary representation
which contains every equivalence class of
irreducible representation \cite{GZ1984}.
This notion has been extended to models of various series
of unitary $G$-representations,
including the $H$-series \cite[\S4]{C2014}.
In Remark \ref{rkm}, we recall this idea and
briefly sketch how Theorem \ref{thm2} leads to models of $H$-series.

We perform symplectic reduction \cite{MW1974}.
Let $\mu \in \mbox{Im}(i\Phi) \subset i\fhh^*$.
There exists a unique $v \in \fhh$ such that
$\Phi^{-1}(\mu) = G \times \{v\}$. Let
\beq \label{ijij}
 \imath : \Phi^{-1}(\mu) \hookrightarrow X \text{ and }
\jmath : \Phi^{-1}(\mu) \lra G/\hH
\eeq
respectively be the natural inclusion and fibration by $\hH$.
Then there is a unique $G$-invariant symplectic form $\om_\mu$
on $G/\hH$ such that $\jmath^* \om_\mu = \imath^* \om$.
Write $X_\mu = G/\hH$. The process
\[ (X, \om) \leadsto (X_\mu, \om_\mu) \]
is called symplectic reduction with respect to $\mu$, and
$(X_\mu, \om_\mu)$ is called the symplectic quotient.

Suppose that $e^\mu \in \wh{H^0}$.
We use the spinor bundle and Dirac cohomology to construct
a unitary representation $\pi_{(X_\mu,\om_\mu)}$ of $G$ in (\ref{yamm}).
Let $\caH_{(X, \om)}$ be the representation space and
$\caH_{(X, \om),\mu}$ its $\mu$-component in the
direct integral decomposition of $\caH_{(X,\om)}$.
The next theorem says that our construction
satisfies the principle {\it quantization commutes with reduction}
proposed by Guillemin and Sternberg \cite{GS1982}.

\begin{theorem}\label{thm3}
Let $G$ be a real reductive Lie group of Harish--Chandra class.
Suppose that $\mu \in \im(i \Phi)$ with $e^\mu \in \wh{H^0}$.
As unitary representations of $G$,
\[
\pi_{(X_\mu, \om_\mu)} \cong \pi_{(X, \om),\mu} \text{ and }
\caH_{(X_\mu, \om_\mu)} \cong \caH_{(X, \om),\mu} .
\]
\end{theorem}

Finally, we show how our results extend from the Harish-Chandra class
of Lie groups $G$ to the class of general real reductive groups
introduced in \cite{W1974a}.
\begin{equation} \label{gen-real-class}
\begin{aligned}
  &\text{(a) the Lie algebra $\fgg$ of $G$ is reductive}, \\
  &\text{(b) if $g \in G$ then $\Ad(g)$ is an inner automorphism of
	$\fgc$, and}\\
  &\text{(c) $G$ has a closed normal abelian subgroup $Z$ such that}\\
      &\phantom{XXii}  (i)\,\, Z \text{ centralizes } G^0\,,
		\text{ i.e. } Z \subset Z_G(G^0) \\
      &\phantom{XXi}  (ii)\,\, |G/ZG^0| < \infty \text{ and } \\
      &\phantom{XX}  (iii)\,\, Z \cap G^0 \text{ is co-compact in } Z_{G^0}\,.
\end{aligned}
\end{equation}
These conditions are inherited by reductive components of cuspidal
parabolic subgroups, and the class of groups that satisfy them includes
both Harish--Chandra's class and all connected real semisimple Lie
groups.  See \cite[\S 0.3]{W1974a} for details.
\vskip 0.5cm

\section{A Plancherel decomposition for $L^2(G/N)$}\label{sec2}
\setcounter{equation}{0}

In this section, we obtain a Plancherel decomposition of $L^2(G/N)$
as $G \times M$ module.  More precisely, we show that
parabolic induction from $M$ to $G$ on the left side
Plancherel decomposition of $L^2(M)$ (as a  $M\times M$ module)
gives a  $G\times M$ module  Plancherel decomposition of $L^2(G/N)$.

We first recall the general setting of induced representations.
Let $X$ be a Lie group.
Let $dx$ be the left invariant measure on $X$, and it is unique up to
positive scalar multiplication.
For each $x \in X$, we let $R_x$ be the right action by $x$.
So $R_x dx$ is again a left invariant measure on $X$, and
there exists a positive number $\delta_X(x)$ such that
$R_x dx = \delta_X(x) dx$. The resulting multiplicative
group homomorphism
\[ \delta_X : X \lra \br^\times \]
is the {\em modular function of} $X$.
If $X$ is abelian, discrete, compact, nilpotent
or reductive, then it is {\em unimodular}, namely $\delta_X \equiv 1$.

Let $Y$ be a closed subgroup of $X$.  Define
\[ \de_Y^X : Y \lra \br^\times \quad \text{by} \quad
\de_Y^X(y) = \delta_Y(y)^{1/2} \delta_X(y)^{-1/2} .\]
If $\eta \in \widehat{Y}$ then the {\em (unitarily) induced
representation} $\pi = {\ind}_Y^X(\eta)$ is the natural left
translation action of $X$ on the space $\ch_\pi$ given by
\[
\{ f: X \ra \ch_\eta \;\mid\;
f(xy) = \de_Y^X(y) \pi_y (f(x)) \mbox{ for all } x \in X, y \in Y
\text{ and } ||f|| \in L^2(X/Y,\ch\}.
\]
Note that $\de_Y^X$ compensates any failure of left Haar measure on $X$ to
define an $X$--invariant Radon measure on $X/Y$.  Of course
\beq \label{indu}
\text{if
$\de_Y^X \equiv 1$, i.e. if $\delta_X|_Y = \delta_Y$, then
$\ch_\pi \cong L^2(X/Y) \widehat{\otimes} \ch_\eta$}\,.
\eeq

We now consider our setting, namely $G,\,K,\,M,\,A,\,N,\,T,\,P \text{ and }H$
as given in the Introduction.
Consider the direct product $G \times M$, with subgroups
\[ M_{diag} \subset NM_{diag} \subset G \times M. \]
Here $M_{diag}$ is the diagonal subgroup isomorphic to $M$, so
$NM_{diag}$ consists of $(nm,m)$ where $n \in N$ and $m \in M$.
Note that $NM_{diag}$ is a well defined subgroup of $G \times M$ because
$M$ normalizes $N$.
We write the elements of the quotient $(G \times M)/NM_{diag}$
as $[g,m]$. There is an action of $G \times M$ on $(G \times M)/NM_{diag}$,
where the right action of $M$ is given by
$R_{m_1}[g,m] = [g,m_1m]$ for all $m_1 \in M$ and
$[g,m] \in (G \times M)/NM_{diag}$.

Consider the exponential map $\exp: \fa \to A$.
As usual $\rho_\fa \in \fa^*$ denotes half the sum of positive $\fa$--roots
determined by $N$, so $2\rho_\fa(v) = \tr(\ad(v): \fn \to \fn)$ for
$v \in \fa$.  It defines the quasi--character
\[ e^{\rho_\fa} : A \to \br^\times \text{ by } e^{\rho_\fa}(\exp(v))=e^{\rho_\fa(v)}
\mbox{ for all } v \in \fa . \]

\begin{proposition} \label{ide}
We have a $(G \times M$)--equivariant diffeomorphism
\[ \phi : (G \times M)/NM_{diag} \lra G/N
\text{ defined by } \phi ([g,m]) = gm^{-1}N ,\]
and it equips $(G \times M)/NM_{diag}$ with a $(G \times M)$--invariant measure.
\end{proposition}
\begin{proof}
Since $M$ normalizes $N$, we have a transitive action $\tau$ of
$G \times M$ on $G/N$, given by $\tau_{(g,m)} xN = gxm^{-1}N$.
We have $\tau_{(g,m)} eN = eN$ if and only if $(g,m) \in NM_{diag}$,
so the stabilizer of the identity coset $eN \in G/N$ is $NM_{diag}$.
This leads to the $(G \times M)$-equivariant diffeomorphism $\phi$
of this proposition.

We have $G = KMAN$.
Let $dg,\, dk,\, dm,\, da \text{ and }dn$ be their Haar measures.
Then $dg = e^{2\rho_\fa} \, dk \, dm \, da \, dn$
\cite[Proposition 8.44]{K2002}.
The $G$-invariant measure on $G/N$ is
$e^{2\rho_\fa} \, dk \, dm \, da$, and it
is invariant under the right action of $M$.
Therefore, $\phi$ induces a $(G \times M)$--invariant measure on
$(G \times M)/NM_{diag}$.
\end{proof}

Since $G$ and $M$ are unimodular, their direct product $G \times M$ is also
unimodular. The left invariant measure on $P = MAN$ is
\beq \label{tid}
e^{2\rho_\fa} \, dm \, da \, dn .
\eeq
The subgroup $NM_{diag}$ of $G \times M$ is isomorphic to the subgroup
$NM$ of $G$.
By (\ref{tid}), $NM$ has Haar measure $dn \, dm$, so it is unimodular.
 Hence $\de_{NM_{diag}}^{G \times M} \equiv 1$, and by (\ref{indu}),
\beq \label{bag}
{\ind}_{NM_{diag}}^{G \times M}(1) = L^2 ((G \times M)/NM_{diag}) .
\eeq

The modular function of $P$ is $\delta_P (man) = e^{2\rho_\fa}(a)$
\cite[VIII-4]{K2002}, so
its restriction to $MN$ is trivial. Therefore,
\beq \label{lil}
\de_{MN \times M}^{P \times M} \equiv 1 \text{ and }
\de_{NM_{diag}}^{MN \times M} \equiv 1 .
\eeq
We apply induction in stages and get
\beq
\begin{array}{rll}
	{\ind}_{NM_{diag}}^{G \times M}(1)
& = {\ind}_{P \times M}^{G \times M} {\ind}_{MN \times M}^{P \times M}
	{\ind}_{NM_{diag}}^{MN \times M}(1) & \mbox{by \cite[Thm.2.47]{KT2013}} \\
& = {\ind}_{P \times M}^{G \times M} {\ind}_{MN \times M}^{P \times M}
(L^2(M) \otimes 1)  & \mbox{by (\ref{indu}) and (\ref{lil})} \\
& = {\ind}_{P \times M}^{G \times M} (L^2(A,L^2(M)) \otimes 1).
 & \mbox{by (\ref{indu}) and (\ref{lil})}
\end{array}
\label{maa}
\eeq

By the Plancherel theorem (see for example \cite[Thm.7.9]{R1973}),
\[
L^2(A: L^2(M)) =
\int_{\sigma\in\fa^*} L^2(M) \otimes e^{i\sigma} \, d\sigma ,
\]
where $d \sigma$ is Lebesgue measure on $\fa^*$. Therefore,
(\ref{maa}) becomes
\beq
{\ind}_{NM_{diag}}^{G \times M}(1) =
\int_{\sigma\in\fa^*} {\ind}_{P \times M}^{G \times M}
(L^2(M) \otimes e^{i\sigma} \otimes 1) d\sigma .
\label{mcc}
\eeq

\begin{proposition} \label{sat}
As the Hilbert spaces for unitary representations of $G \times M$,
\[
L^2(G/N) = \int_{\sigma\in\fa^*}
{\ind}_{P \times M}^{G\times M} (L^2(M)\otimes e^{i\sigma} \otimes 1)\, d\sigma .
\]
\end{proposition}
\begin{proof}
This follows from
Proposition \ref{ide}, (\ref{bag}) and (\ref{mcc}).
\end{proof}

\section{Partial Dirac cohomology of $L^2(G/N)$}\label{sec3}
\setcounter{equation}{0}

In this section we calculate the partial Dirac cohomology for the
representation of $G \times M$ on $L^2(G/N)$ with
respect to 
$D_M = D_{(\fmc,\fmc\cap\ftc)} + iD_{diag;(\fmc\cap\fkc,\ftc)}$,
the modified Dirac operator
of (\ref{mdirac}).  We then apply it to prove Theorem \ref{thm1}.

If $N=1$ and $M=G$, it reduces to
the calculation of Dirac cohomology of $L^2(G)$
with respect to the Dirac operator $\widetilde{D}_{(\fgc,\ftc)}$\,.
That was done in \cite{CH2016}.
Therefore, we use the phrase ``partial Dirac cohomology'' for
the case when $N\neq 1$.
The essential part is the calculation of the
Dirac cohomology of discrete series representations of $M$.

Now we recall the relevant results of Dirac cohomology of discrete series
representations.  Let $M$ be an acceptable real reductive Lie group
of Harish--Chandra class, for example the group $M$ in the Iwasawa
decomposition $P = MAN$ of a cuspidal parabolic subgroup of our group $G$.
We described the discrete series $\wh{M}_{disc}$ of $M$ in the
discussion leading up to (\ref{def-ds-M}).  It consists of the representations
$\eta_{\chi,\lambda} = {\ind}_{M^\dagger}^M(\eta_{\chi,\lambda}^\dagger)$
specified there.  The parameterization is that $T^0$ is a compact Cartan
subgroup of $M^0$\,, $\lambda$ belongs to the $M^0$--regular subset of
the lattice $\Lambda = \{\nu \in \ft^* \mid e^{i\nu} \in \widehat{T^0}\}$,
and $\chi \in Z_M(M^0)$ agrees with $e^{\lambda +\rho_\fm}$ on $Z_{M^0}$\,.
Harish--Chandra's construction and characterization of the discrete
series (for acceptable groups of Harish--Chandra class) is

\begin{theorem} \label{thm.4.1}
The discrete series $\widehat{M}_{disc}$ consists of the equivalence classes
of representations $\eta_{\chi,\lambda}$ where $\lambda$ runs over the
set of $M$--regular elements in the lattice $\Lambda \subset \ft^*$
and $\chi \in \wh{Z_M(M^0)}$ agrees with $e^{\lambda + \rho_\fm}$ on the
center of $M^0$\,.  Discrete series representations
$\eta_{\chi,\lambda} \simeq \eta_{\chi',\lambda'}$ if and only if
$(\chi,\lambda)$ and $(\chi',\lambda')$ are in the same orbit of the
Weyl group $W_K$\,.
\end{theorem}

We denote by $V_{\chi,\lambda}$ the Harish--Chandra module of
the $(K\cap M)$--finite vectors in $\caH_{\eta_{\chi,\lambda}}$.
The elements of $V_{\chi,\lambda}$ are $C^\infty$ vectors (in
fact real analytic vectors), and $V_{\chi,\lambda}$ is dense in
$\caH_{\eta_{\chi,\lambda}}$.  Thus the modified Dirac operator
$D=\Dt_{(\fgc,\ftc)}$ is a densely defined symmetric operator
$$
D\colon \caH_{\eta_{\chi,\lambda}} \otimes S \rightarrow
	\caH_{\eta_{\chi,\lambda}} \otimes S.
$$
We recall a standard fact in functional analysis that
a densely defined symmetric operator is closable and its closure is also symmetric.
If $A^*$ denotes the adjoint of a densely defined symmetric operator
$A$, then the closure $c\ell(A) = (A^*)^*$ and
$c\ell(A)$ is also symmetric \cite[Lemma 20.1]{MV1997}.
Thus, $D$ is closable and its closure $c\ell(D)$  is also symmetric.
It follows that ${\rm Ker}\, c\ell(D)$ is a closed subspace of the
Hilbert space $\caH_{\eta_{\chi,\lambda}} \otimes S$.  We define the
Dirac cohomology $H_D(\caH_{\eta_{\chi,\lambda}})$ of an irreducible
unitary representation $\eta_{\chi,\lambda}$
to be ${\rm Ker\,} c\ell(D)$.  The following proposition shows that
the Dirac cohomology of an irreducible representation is
equal to the Dirac cohomology of its Harish-Chandra module.
It was proved in \cite[Prop. 3.2]{CH2016} for
the case of connected semisimple Lie groups of finite center.

\begin{proposition}\label{dbar}
The kernel of $c\ell(D) \colon \caH_{\eta_{\chi,\lambda}} \otimes S
\rightarrow \caH_{\eta_{\chi,\lambda}} \otimes S$ coincides with the
kernel of $D \colon V_{\chi,\lambda} \otimes S\rightarrow
V_{\chi,\lambda} \otimes S$.  Thus
$$
\mbox{\rm Ker}\, c\ell(D) =
	H_D(V_{\chi,\lambda})
      = {\sum}_{w\in W_{K\cap M}}
              \ch_{\chi\cdot w^{-1}} \otimes  \bbC_{w\lambda}.
$$
\end{proposition}

\begin{proof} If $M$ is connected and semisimple the assertion is
\cite[Prop. 3.2]{CH2016}.  But the argument of \cite[Prop. 3.2]{CH2016}
goes through without change, and without requiring semisimplicity because
$\eta_\lambda^0$ restricts to the center $Z_{M^0}$ as a multiple of
some fixed unitary character $\zeta$.  Thus the assertion holds for
connected $M$.

Consider the case $M = Z_M(M^0)M^0$.  The argument of \cite[Prop. 3.2]{CH2016}
still shows that ${\rm Ker\,} c\ell(D) = {\rm Ker\,} D$, and conjugation by
elements of $Z_M(M^0)$ makes no change in $D$, so $H_D(V_{\chi,\lambda})
= {\sum}_{w\in W_{K\cap M^0}} \ch_{\chi\cdot w^{-1}} \otimes \bbC_{w\lambda}$.
That gives us the assertion for $M = M^\dagger$.

Finally consider the general case.  There $\eta_{\chi,\lambda}|_{M^\dagger}
= \sum_{xM^\dagger \in M/M^\dagger}
\eta^\dagger_{\chi\cdot \Ad(x),\lambda \cdot \Ad(x)}$ because
$\eta_{\chi,\lambda}$ is induced from the normal subgroup $M^\dagger$.
We use the result for $M^\dagger$ to write this as
$\eta_{\chi,\lambda}|_{M^\dagger} =
\sum_{w\in W_{K\cap M}} \ch_{\chi\cdot w^{-1}} \otimes  \bbC_{w\lambda}$.
The assertion follows.
\end{proof}

Recall that the Dirac operator $D$ is in $U(\fmc)\otimes C(\fsc)$.
We first consider
$$
D \colon C^\infty(M)\otimes S \rightarrow
C^\infty(M)\otimes S
$$
Then $D$ induces a densely defined symmetric
operator on the Hilbert space $L^2(M)\otimes S$, and the closure of
$D$ defines a closed symmetric operator
\[
\bbD \colon L^2(M)\otimes S \rightarrow
L^2(M)\otimes S.
\]
Then ${\rm Ker}\, \bbD$ is a closed subspace in
$L^2(M)\otimes S$.
We define the Dirac cohomology $H_D(L^2(M))$ of $L^2(M)$ to be
${\rm Ker}\, {\bbD}$.
It follows from  the fact that $D$ is $T$-invariant that
${\rm Ker}\, {\bbD}$ is
a $(M\times T)$--module.  The following theorem was proved for the case
of connected $M$ as \cite[Theorem 3.3]{CH2016}.  Since $M$ is of
Harish--Chandra class, the argument there goes through to prove
${\rm Ker}\, {\bbD}=\sum_{\eta_{\chi,\lambda}\in\widehat{M}_{disc}}
\mathcal{H}_{\eta_{\chi,\lambda}} \otimes H_D(V_{\chi,\lambda}^*)$.
Using Proposition \ref{dbar}, now,
\begin{theorem} \label{key}
We have the following orthogonal sum decomposition as representation
space of $(M\times T)$:
$$
{\rm Ker}\, {\bbD}= {\sum}_{\eta_{\chi,\lambda}\in\widehat{M}_{disc}}
\mathcal{H}_{\eta_{\chi,\lambda}} \otimes
	{\sum}_{w \in W_{K\cap M}}(\mathcal H_{\chi^*\cdot w^{-1}}\otimes
	\bc_{-w\cdot\lambda}).
$$
\end{theorem}

Now we calculate the partial Dirac cohomology of $L^2(G/N)$.
We are working with a Cartan involution $\theta$ of $G$,
$K = G^\theta$ is a maximal compact subgroup, $H = T \times A$
is a $\theta$--stable Cartan subgroup, $P = MAN$ is an associated
cuspidal parabolic subgroup so $M$ has compact Cartan subgroup
$T$ and $MA = Z_G(A)$, and $G = KMAN$.

Let $\fmc = \ftc \oplus \fsc$ where $\fsc$ is the sum of the $\ftc$--root
spaces in $\fmc$.  It is the orthogonal decomposition of $\fmc$
with respect to an invariant form of $\fg$.
Let $S$ be the spin module for the Clifford algebra $C(\fsc)$.
Then
$$
D_M = D_{\fmc, \fkc \cap \fmc} + i D_{\delta;\fkc \cap \fmc, \ftc}
$$
is the {\em modified Dirac operator} as defined in \cite[(1.1)]{CH2016}.
It induces a densely defined symmetric operator,
$$
 \bd_M : L^2(G/N) \otimes S \lra L^2(G/N) \otimes S .
$$

\noindent {\it Proof of Theorem \ref{thm1}:}

By Proposition \ref{sat}, the representation space of the regular
representation of $G\times M$ on $L^2(G/N)$ has direct integral decomposition
$$
L^2(G/N) = \int_{\sigma\in\fa^*} {\ind}_{P \times M}^{G \times M} (L^2(M)
\otimes e^{i\sigma} \otimes 1) \, d\sigma .
$$
This
decomposition splits into  a discrete spectrum and a continuous spectrum.
The discrete part corresponds to
the discrete spectrum of the Plancherel decomposition of $L^2(M)$
with summation over all discrete series of $M$,
$$
L^2(G/N)_{disc}=
  {\sum}_{\eta_{\chi,\lambda}\in\widehat M_{disc}}
  \left ( \int_{\sigma\in\fa^*} {\ind}_{P }^{G}
  \bigl ( (\caH_{\eta_{\chi,\lambda}}\otimes \caH_{\eta_{\chi,\lambda}}^*)
  \otimes e^{i\sigma} \otimes 1\bigr ) \, d\sigma \right )
  \deg(\eta_{\chi,\lambda})\,.
$$
Here $\deg(\eta_{\chi,\lambda})$ is the formal degree of
$\eta_{\chi,\lambda}$\,.
The continuous spectrum corresponds to a direct integral
of other tempered (i.e. $H$--series) representations of $M$.

\begin{lemma}\label{dsDM}
Only the discrete spectrum in $L^2(M)$ contributes to
${\rm Ker}\, \bbD_M $.
\end{lemma}
\begin{proof}
It follows from the Plancherel decomposition that
$$
{\rm Ker}\, \bbD_M= \int_{\eta\in\widehat M}
(\int_{\sigma\in\fa^*}{\ind}_{P }^{G}
(\caH_\tau \otimes e^{i\sigma} \otimes 1)\, d\sigma)
\otimes {\rm Ker}\, \{\overline{D_M}\colon
\caH_\eta^*\otimes S\rightarrow \caH_\eta^*\otimes S\}d\mu(\eta)
$$
where $\mu$ is Plancherel measure on $\wh{M}$.
Therefore, the representation $\eta \otimes \eta^*$ of $M\times M$
contributes to ${\rm Ker}\, \bbD_M$ if and only if the Dirac operator
$\overline{D_M}$ acts on $\caH_\eta^*\otimes S$ with nonzero kernel.
This condition is
equivalent to the Harish-Chandra module of $\caH_\eta^*$ having
nonzero Dirac cohomology, as we showed in Proposition
\ref{dbar}. It follows that $\caH_\eta^*$ must have a real
infinitesimal character $\xi$, in the sense that if $\xi$ is in the
dominant chamber then $\theta(\xi)=\xi$. Therefore, only the discrete
series can contribute to ${\rm Ker}\,\bbD_M$, since other tempered
representations with real infinitesimal characters have Plancherel
measure 0 by Harish-Chandra's Plancherel Theorem.
\end{proof}

From Lemma \ref{dsDM} we obtain the orthogonal sum decomposition
$$
{\rm Ker}\,\bbD_M=\sum_{\eta_{\chi,\lambda}\in\widehat
M_{disc}}\left ( \int_{\sigma\in\fa^*} {\ind}_{P }^{G}
	(\caH_{\eta_{\chi,\lambda}}
	\otimes e^{i\sigma} \otimes 1)\, d\sigma\right )
	\deg(\eta_{\chi,\lambda}) \otimes H_{D}(\caH_{\eta_{\chi,\lambda}}^*).
$$
Then by substituting the Dirac cohomology $H_{D}(\caH_\lambda^*)$ of the
discrete series representation $\caH_{\eta_{\chi,\lambda}}^*$ of $M$,  we obtain
\begin{equation}\label{kerdm}
\begin{aligned}
{\rm Ker}\,\bbD_M=  \sum_{\eta_{\chi,\lambda}\in\widehat
M_{disc}}& \left ( \int_{\sigma\in\fa^*} {\ind}_{P }^{G}
        (\caH_{\eta_{\chi,\lambda}}
	\otimes e^{i\sigma} \otimes 1)\, d\sigma\right )
        \deg(\eta_{\chi,\lambda}) \\
& \otimes \bigl ( \sum_{w\in W_{K \cap M}}
	\caH_{\chi^* \cdot w^{-1}}\otimes \bbC_{-w\lambda}\bigr ).
\end{aligned}
\end{equation}
This completes the proof of Theorem \ref{thm1}.  $\hfill$$\Box$

We note that
$\bbD_M$ is $G\times T$-invariant, so (\ref{kerdm})
is an orthogonal decomposition of  $(G\times T)$--modules.

\section{$L^2$-functions}\label{sec4}
\setcounter{equation}{0}

In this section, we study certain $L^2$-function spaces.
The key result is Proposition \ref{keyy}, which will be used later.

Recall some notation.  We fix a Cartan involution $\theta$ of $G$ and  
the $(\pm 1)$--eigenspace decomposition $\fg = \fk + \fp$; $\fk$ is the
Lie algebra of the fixed point set $K = G^\theta$.  We fix a 
$\theta$--stable Cartan subgroup $H = T \times A$ where $T = H\cap K$
and $A = \exp(\fa)$,\, $\fa = \fh \cap \fp$.  
Then $H^0 = T^0\times \exp(\fa) \cong T^0 \times \fa$, 
and $i\ft \cong \exp(i\ft)$ (group isomorphisms), and
$\fh = \ft + \fa \cong i\fh = i\ft \times i\fa$ (real vector
space isomorphism),
leading to identification of
$\tT_\bc = T^0 \times \exp(i\ft)$ with $(\bc/\bz)^n$, $n = \dim_\br T$, 
and of $\fa_\bc$ with $\bc^m$, $m = \dim_\br A$.  

From that, there is a unique complex 
structure on $\hH \times \fhh$ such that the map
\beq
\label{dapat}
(t \exp v, x+y) \mapsto (t \exp ix , v + iy)
\text{ for all } t \in T^0, x \in \ft \text{ and } v,y \in \fa
\eeq
is a holomorphic diffeomorphism of $\hH \times \fhh$ onto
$\tT_\bc \times \fa_\bc \cong (\bc/\bz)^n \times \bc^m$\,.  This uses
$t \exp v \in H^0$, $x+y \in \fh$, $t \exp ix \in T_\bc^0$ 
and $v+iy \in \fa_\bc$\,.

Let $\caH_{(\hH \times \fh)}$ denote the resulting space of holomorphic 
functions
on $\hH \times \fh$.  The space $\widehat{H^0}$ of unitary characters on 
$H^0$ consists of the $e^\mu$, $\mu \in i\fh^*$, that are well defined on
$H^0$.  Any such $e^\mu$ extends uniquely to a holomorphic homomorphism
\[ 
e^\mu_\bc: \hH \times \fh \to \bc^\times,
\] 
and $e^\mu_\bc \in \caH_{(\hH \times \fh)}$.  We write
$\bc_\mu$  for the $1$--dimensional space spanned by 
$e^\mu_\bc$.

Fix a strictly convex function
\[ F : \fhh \lra \br .\]
Namely $F$ is a smooth function
such that under any linear coordinates $(x_i)$ on $\fhh$,
the Hessian matrix $(\frac{\pa^2 F}{\pa x_i \pa x_j})$
is positive definite.
We also identify it with an $\hH$-invariant function on $\hH \times \fh$,
and a $G$-invariant function on $G/N \times \fh$.
For functions on $\fh$, $\hH \times \fh$ and $G/N \times \fh$,
the $L^2$-norm $\|\cdot\|^2$ refers to
square-integration against $e^{-F}$ times invariant measure.
For instance we define the weighted Bergman space
\[ 
\caHH_{(\hH \times \fhh, e^{-F})} =
\left \{f \in \caH_{(\hH \times \fh)} \left |
\int_{\hH \times \fhh} |f(h,x)|^2 e^{-F(x)} dh\, dx <\infty \right . \right \} ,
\]
where $dx$ is the Lebesgue measure on $\fh$.
The holomorphic functions form a closed subspace in the $L^2$-space,
so $\caHH_{(\hH \times \fhh, e^{-F})}$ is a Hilbert space.

Let $F' : \fhh \lra \fhh^*$
be the gradient mapping of $F$.
The image $\im(\hi F') \subset i\fhh^*$.
(The factor $i$ is added so that the image lies in $i \fhh^*$.)
Since $F$ is strictly convex, $\im(\hi F')$ is a convex open set.
Let $d \mu$ be the product of counting measure on
$\{\lambda \in i\ft^* \mid e^\lambda \in \widehat{T^0}\}$
and Lebesgue measure on $i \fa^*$.  It is a normalization
of Haar measure on $\widehat{H^0}$.

\begin{theorem} {\rm \cite[Thm.1.2]{C2007}}
We have an isomorphism of unitary $\hH$-representations
\label{cied}
$$
\caHH_{(\hH \times \fhh, e^{-F})} =
\int_{\tiny \begin{array}{l} \mu \in \im(\hi F') \\
e^\mu \in \wh{\hH} \end{array}} \bc_\mu \, d \mu .
$$
\end{theorem}

We next study functions on $Y = G/N \times \fhh$ .
There is a natural embedding and a fibration
\begin{equation}
 \imath : G/N \hookrightarrow Y \;\;,\;\;
\pi : Y \lra G/\hH N .
\label{tadi}
\end{equation}
Here $\imath(g) = (g,0)$ for all $g \in G/N$
and $\pi$ is the natural quotient.
Given $f : Y \lra S$, we let $\imath^*f : G/N \lra S$ be
its pullback to $G/N$.
Each fiber of $\pi$ is diffeomorphic to $\hH \times \fhh$,
so by (\ref{dapat}), it has a complex structure.
We say that $f$ is $\pi$--{\em holomorphic} or {\em partially 
holomorphic} if it is holomorphic on each fiber of $\pi$.
It is the same to consider holomorphic properties on the fibers of
$Y \lra G/H^0N$ and $Y \lra G/HN$, because their fibers have the same
connected components. Let
\[  \caH_{(Y)} \otimes S =
 \{f : Y \lra S  \mid 
  \imath^* f \in {\rm Ker}\,\bbD_M \mbox{ and }
f \mbox{ is $\pi$-holomorphic}\} . \]

Recall that $F$ is a strictly convex function on $\fh$. Let
\beq
  \caHH_{(Y, e^{-F})} \otimes S =
 \left \{f \in \caH_{(Y)} \otimes S \left |
 \int_Y |f(g,x)|^2 e^{-F(x)} d(gN) \, dx < \infty \right . \right \} .
 \label{aap}
 \eeq
The right $\hH$-action on $G/N$ leads to the direct integral decomposition
\begin{equation}
{\rm Ker}\,\bbD_M = \int_{e^\mu \in \wh{\hH}} ({\rm Ker}\,\bbD_M)_\mu \, d \mu .
 \label{taa}
 \end{equation}
No single integrand $({\rm Ker}\,\bbD_M)_\mu$ is contained in
${\rm Ker}\,\bbD_M$ (it is only weakly contained) because it has 
$d\mu$--measure 0.
Thus $f_\mu \in ({\rm Ker}\,\bbD_M)_\mu$ is generally not square-integrable
over $G/N$. It transforms by $e^\mu$ under the right $\hH$-action,
namely $R_h^* f_\mu = e^\mu (h) f_\mu$.

\begin{proposition}
The map $\imath^* : \caHH_{(Y, e^{-F})} \otimes S \lra {\rm Ker}\,\bbD_M$
is injective.  It defines a $G \times \tT$--equivariant isomorphism
\[ \caHH_{(Y, e^{-F})} \otimes S \cong
\int_{\mu \in \im(\hi F'), e^\mu \in \wh{\hH}}
({\rm Ker}\,\bbD_M)_\mu \, d \mu .\]
\label{keyy}
\end{proposition}
\begin{proof}
We first show that $\imath^*$ is injective.
Suppose that $f_1, f_2 \in \caHH_{(Y, e^{-F})} \otimes S$ satisfy
$\imath^* f_1 = \imath^* f_2$, namely
$f_1$ and $f_2$ agree on $G/N$.  Fix $g \in G/N$, and we have
\beq
f_1 (gh,0) = f_2(gh,0) \mbox{ for all } h \in \hH .
\label{aapq}
\eeq
Being $\pi$--holomorphic,
$f_1, f_2$ are holomorphic on the fiber $(g H^0, \fh)$ of $\pi$.
So together with (\ref{aapq}), we have
\beq
f_1 (gh,x) = f_2(gh,x)  \mbox{ for all } (h,x) \in \hH \times \fh .
\label{ccpq}
\eeq
Since (\ref{ccpq}) holds for all $g \in G/N$,
it follows that $f_1 = f_2$. So $\imath^*$ is injective.
It is clear that $\imath^*$ intertwines with the action of $G \times \tT$.
It remains to prove that
\begin{equation}
 \imath^*(\caHH_{(Y, e^{-F})} \otimes S) =
 \int_{\tiny \mu \in \im(\hi F'),\, e^\mu \in \wh{\hH}} 
({\rm Ker}\,\bbD_M)_\mu \, d \mu .
 \label{duas}
 \end{equation}

 We first check the $\subset$ part of (\ref{duas}).
 Let $f \in \caHH_{(Y, e^{-F})} \otimes S$.
 Then $\imath^* f \in {\rm Ker}\, \bbD_M$, so
 by (\ref{taa}), we write
 \[ f(g,0) = \int_{e^\mu \in \wh{\hH}}
  f_\mu(g) \, d\mu , \]
 where $f_\mu \in ({\rm Ker}\, \bbD_M)_\mu$ transforms by $e^\mu$
 under the right action of $\hH$.
We claim that
 \begin{equation}
  f(g,x) = \int_{e^\mu \in \wh{\hH}}
  f_\mu(g) e_\bc^{\mu(x)} \, d\mu .
  \label{lars}
  \end{equation}
  The function $f_\mu(g) e_\bc^{\mu(x)}$ transforms by
 $e_\bc^\mu$ under the right action of $\hH \times \fh$,
 so it is holomorphic on the fiber $(g \hH,\fh)$ of $\pi$.
 Hence for each $g \in G/N$,
 both sides of (\ref{lars}) agree on
 $(g \hH,0)$ and are holomorphic on $(g \hH, \fh)$,
 so they agree on $(g \hH,\fhh)$.
 This holds for each $g$,
 which proves (\ref{lars}) as claimed.

The restriction of $f$ to $\hH \times \fh$ belongs to
$\caHH_{(\hH \times \fh, e^{-F})} \otimes S$.
 By Theorem \ref{cied}, it
 is a direct integral over $\{\mu \in \im(\hi F')\}_{e^\mu \in \wh{\hH}}$.
 So in (\ref{lars}), $f_\mu = 0$ for
 $\mu \not\in \im(\hi F')$. This proves the $\subset$ part of (\ref{duas}).

Next we prove the $\supset$ part of (\ref{duas}). Pick
\begin{equation}
f^0 = \int_J f_\mu \, d \mu \in \int_J ({\rm Ker}\,\bbD_M)_\mu \, d \mu
\subset {\rm Ker}\, \bbD_M \subset L^2(G/N) \otimes S ,
\label{qqa}
\end{equation}
where $J \subset \im(\hi F')$ is Borel--measurable 
and $e^J \subset \wh{\hH}$. We may assume that $J$
is bounded, as the members of ${\rm Ker}\, \bbD_M$ are Hilbert space sums
of such elements. Let
\[ 
f: Y \lra S \text{ defined by } f (g,x) = \int_J f_\mu(g) e_\bc^{\mu x} \, d \mu .
\]
Here $(\imath^* f)(g) = f(g,0) = f^0(g)$, so 
$\imath^*f \in {\rm Ker}\, \bbD_M$.
Also, $(g,x) \mapsto f_\mu(g) e_\bc^{\mu(x)}$ is a $\pi$-holomorphic 
function for each $\mu$, so $f$ is $\pi$-holomorphic.
Hence $f \in \caH_{(Y)} \otimes S$.
We want to show that $f^0 \in \imath^*(\caHH_{(Y, e^{-F})} \otimes S)$
in (\ref{duas}), so it remains to check that
\beq
\|f\|^2 = \int_Y |f(g,x)|^2 e^{-F(x)} d(gN) \, dx < \infty .
\label{kami}
\eeq

The condition $\imath^*f \in {\rm Ker}\, \bbD_M$ implies that,
in particular,
\beq
f(\cdot , 0) \in L^2(G/N) \otimes S .
\label{wye}
\eeq
The holomorphic homomorphism $e_\bc^{\mu(x)} : H^0 \times \fh \to \bc^\times$ 
maps the $H^0$ and $\fh$ components to $S^1$ and $\br^\times$ respectively.
Fix $x \in \fh$.
Since $J$ is bounded, the set $\{e_\bc^{2 \mu(x)}\}_{\mu \in J}$
is bounded above by some $m = m_x$.
We have
\beq
\begin{array}{rl}
\|f(\cdot,x)\|^2 &
= \int_J \|f_\mu e_\bc^{\mu(x)}\|^2 d \mu
= \int_J \int_{G/N} |f_\mu(g)|^2  e_\bc^{2 \mu(x)} d(gN) \, d \mu \\
& \leq m \int_J \int_{G/N} |f_\mu(g)|^2  d(gN) \, d \mu
= m  \int_J \|f_\mu\|^2 d \mu = m \|f(\cdot,0)\|^2 .
\end{array} \label{jjaa}
\eeq
By (\ref{wye}) and (\ref{jjaa}), for all $x \in \fh$,
\beq
f(\cdot , x) \in L^2(G/N) \otimes S .
\label{wyf}
\eeq

Let $\caH_{(\fh)}$ denote the analytic functions on $\fh$.
By (\ref{wyf}), we can define
\beq
\caH_{(Y)} \otimes S \lra \caH_{(\fh)} \otimes L^2(G/N) \otimes S
\text{ by } f \mapsto \wt{f} ,\label{oren}
\eeq
where $\wt{f}(x) = f(\cdot, x)$ for all $x \in \fh$. Let
\[ \caHH_{(\fh, e^{-F})} \otimes L^2(G/N) \otimes S
= \left \{k \in \caH_{(\fh)} \otimes L^2(G/N) \otimes S \left |
\int_\fh |k(x)|^2 e^{-F(x)} dx < \infty \right . \right \}.\]
For $f \in \caH_{(Y)} \otimes S$, we have
\beq
 \begin{array}{rl}
\|f\|^2 & = \int_Y |f(g,x)|^2 e^{-F(x)} d(gN) \, dx
= \int_\fh (\int_{G/N} |f(g,x)|^2 \, d(gN)) e^{-F(x)} dx \\
& = \int_\fh  \|\wt{f}(x)\|^2 e^{-F(x)} dx
= \|\wt{f}\|^2 .
\end{array}
\label{shee}
\eeq
Hence (\ref{oren}) leads to a norm preserving map
\[ \caHH_{(Y, e^{-F})} \otimes S \lra
\caHH_{(\fh, e^{-F})} \otimes L^2(G/N) \otimes S .\]

Write $\wt{f} = k \otimes v$, where $k \in \caH_{(\fh)}$ and
$v \in L^2(G/N) \otimes S$.
Since (\ref{oren}) intertwines with the right $\fh$-action,
we have $k = \int_J k_\mu \, d \mu$, where $k_\mu \in \bc_\mu$.
By Theorem \ref{cied}, the function
$(h,x) \mapsto \int_J k_\mu(x) e^\mu (h) \, d \mu$ is square integrable
over $\int_{H^0 \times \fh} (\cdot) e^{-F(x)} dh \, dx$
because $J \subset \im(\hi F')$. So
$k$ is square integrable over
$\int_\fh (\cdot) e^{-F(x)} dx$.
It implies that
$\|\wt{f}\| < \infty$, and hence $\|f\| < \infty$ by (\ref{shee}).
We have proved (\ref{kami}), and therefore
$f^0 = \imath^* f \in \imath^*(\caH^2_{(Y, e^{-F})} \otimes S)$
in (\ref{qqa}). This proves the $\supset$ part of (\ref{duas}).
The proposition follows.
\end{proof}

\section{Geometric quantization}\label{sec5}
\setcounter{equation}{0}

In this section, we incorporate partial Dirac cohomology
into symplectic geometry, and prove Theorem \ref{thm2}.
The intended symplectic manifold is
\[ X = G \times \fhh .\]
We first recall some results from \cite[\S3]{C2014}
on the symplectic geometry of $X$.

Let $\Omega^\bullet$ denote the de Rham complex of differential forms.
Superscript denotes group invariance.
So for example $\Omega^1(\hH \times \fhh)^{\hH}$
consists of the $\hH$-invariant 1-forms on $\hH \times \fhh$.

Let $F : \fhh \lra \br$ be a smooth function,
and let $F'$ be its gradient map.
We may also regard $F$ as a function on $\hH \times \fhh$
or $G \times \fhh$ by invariance on the first component.
As in (\ref{dapat}) and (\ref{tadi}), each fiber of
\begin{equation}
 \pi : X \lra G/\hH
 \label{part}
 \end{equation}
inherits a complex structure
from $\hH_\bc = \hH \times \fh$.
As usual $\pa$ and $\bar{\pa}$ denote its Dolbeault operators. Let
\[ 
\be = - \tfrac{i}{2} (\pa - \bar{\pa})F \in \Omega^1(\hH \times \fh)^{\hH} .
\]
Let $\fh^* \hookrightarrow \fgg^*$ be the inclusion
whose image consists of all linear functionals on $\fg$ that annihilate 
all the root spaces of $\fhh$.  It leads to the inclusion
\[ 
\jmath : \Omega^1(\hH \times \fhh)^{\hH} \hookrightarrow \Omega^1(X)^G .
\]

\begin{theorem} {\rm \cite[Thm.3.1]{C2014}}
The 2-form $\omega = d (\jmath \be) \in \Omega^2(X)^{G \times \hH}$
is symplectic if and only if $F'$ is a local diffeomorphism
and its image $\mbox{\rm Im}(F') \subset \fhh_{\rm reg}^*$.
\label{assam}
\end{theorem}

Here $\fhh_{\rm reg}^*$ consists of the elements of
$\fhh^*$ that are not perpendicular to any root.
We shall fix a strictly convex function $F$ whose gradient has image
$\mbox{Im}(F') \subset \fhh_{\rm reg}^*$.
The strictly convex condition implies that
$F'$ is a local diffeomorphism, so the $2$--form $\omega$ constructed above
is symplectic.

The $G \times \hH$-action on $X$ preserves $\om$ and is Hamiltonian,
and the right $\hH$-action has a canonical moment map \cite[\S11]{GS1984}
\begin{equation}
 \Phi : X \lra \fhh^* \text{ given by }
\Phi(g,x) = \tfrac{1}{2} F'(x)
\label{cano}
\end{equation}
for all $(g,x) \in G \times \fhh = X$ \cite[Prop.3.2]{C2014}.
The conventions in here and \cite{C2014} differ by a factor $2$,
namely \cite{C2014} uses $\be = - i(\pa - \bar{\pa})F$
and $\Phi(g,x) = F'(x)$.

We next perform geometric quantization \cite{K1970} on
the symplectic manifold $(X,\om)$.
There is a complex line bundle $\blx \to X$
whose Chern class of $\blx$ is the cohomology class $[\om]$.
The construction in Theorem \ref{assam} shows that
$\om$ is exact, so $[\om]=0$ and $\blx$ is topologically trivial.
However, $\blx$ has interesting geometry, as it has a connection
$\nabla$ whose curvature is $\om$, as well as
an invariant Hermitian structure.
If $W \subset X$ is a submanifold with a complex structure,
we say that a section $f$ of $\bl$ is holomorphic on $W$
if $\nabla_\xi f$ vanishes on $W$ whenever $\xi$ is an anti-holomorphic
vector field on $W$. Each fiber of $\pi$ of (\ref{part})
is complex, and we say that $f$ is $\pi$--{\em holomorphic} or
{\em partially holomorphic} if it is
holomorphic on each fiber of $\pi$.

There is a natural action of $G \times G$ on $X$,
given by the left and right actions on the $G$-component.
It lifts to a representation of $G \times G$ on the sections of $\blx \to X$.

\begin{proposition} \label {spee}
{\rm \cite[Cor.3.4]{C2014}}
There exists a $(G \times G)$--invariant
non-vanishing section $f_0$ of $\blx$
which is $\pi$-holomorphic, and $(f_0, f_0)(g,x) = e^{-F(x)}$
for all $(g,x) \in G \times \fh = X$.
\end{proposition}

Recall $Y = G/N \times \fhh$.
If a section $f$ of $\bl$ is invariant under the right action of 
$N$, then so is $(f,f)$, and we identify $(f,f)$ with a function on $Y$.
Let
\beq
 L^2(\bl)^N =
\left \{ \mbox{$N$-invariant sections $f$ of $\bl$} \left |
\int_{G/N \times \fh} (f,f)(g,x) \, d(gN) \, dx < \infty \right . \right \} .
\label{twon}
\eeq
Then the action of $G \times \tT$ on $L^2(\blx)^N$ is a unitary representation.

Let $f_0$ denote the
section in Proposition \ref{spee}.
For all $f \in C^\infty (G \times \fh)^N$, we have
\[ \int_Y (ff_0,ff_0)(g,x) \, d(gN) \, dx
= \int_Y |f_(g,x)|^2 e^{-F(x)} d(gN) \, dx ,\]
so the trivialization $ff_0 \mapsto f$ defines a
$(G \times \tT)$--equivariant isometry
\beq
L^2(\bl)^N \cong L^2(Y, e^{-F}) .
\label{aba}
\eeq

Let $\imath$ denote both embeddings $G \hookrightarrow X$
and $G/N \hookrightarrow Y$, where $\imath(g) = (g,0)$.
So $\imath^* \bl$ is a line bundle on $G$.
Since $f_0$ is $G$-invariant, we can normalize it so that
$(f_0,f_0)(g,0)=1$ for all $g \in G$. Then for all $f \in C^\infty(G)^N$,
\[ \int_{G/N} (f(\imath^*f_0),f(\imath^*f_0))(g) \, d(gN)
= \int_{G/N} |f(g)|^2 \, d(gN) ,\]
so the trivialization $ff_0 \mapsto f$ leads to an isometry
\beq
 L^2(\imath^* \bl)^N \otimes S \cong L^2(G/N) \otimes S.
 \label{gme}
 \eeq

Since (\ref{gme}) is $(G \times M)$--equivariant,
it induces an operator
 \[ \bbD_\bl : L^2(\imath^* \bl)^N \otimes S \lra L^2(\imath^* \bl)^N \otimes S \]
 such that (\ref{gme}) intertwines $\bbD_\bl$ and $\bbD_M$. Let
\[ 
\caH^2(\bl)^N \otimes S = \{f \in L^2(\bl)^N \otimes S \mid
\imath^* f \in {\rm Ker}\, \bbD_\bl \mbox{ and $f$ is $\pi$--holomorphic}\}.
\]
Since (\ref{gme}) also preserves the $\pi$--holomorphic property,
together with (\ref{aba}), they imply 
\beq
\caH^2(\bl)^N \otimes S \cong \caH^2_{(Y, e^{-F})} \otimes S .
\label{sees}
\eeq

\noindent {\it Proof of Theorem \ref{thm2}:}

By Proposition \ref{keyy}, (\ref{cano}) and (\ref{sees}),
\begin{equation}
\caH_{(X, \om)} \cong \caH^2(\blx)^N \otimes S
\cong \caH^2_{(Y, e^{-F})} \otimes S
\cong \int_{\tiny \begin{array}{l}
\mu \in \im(i \Phi) \\
e^\mu \in \wh{\hH} \end{array}} ({\rm Ker}\,\bbD_M)_\mu \, d \mu .
\label{hap}
\end{equation}

Write $\mu = \la + i \si \in \im(i \Phi)$, where $e^\la \in \wh{\tT}$.
By Theorem \ref{assam} and (\ref{cano}), $\mu$ is $MA$--regular.
Let $\rho_\fa$ be half the sum of positive $\fa$--roots relative to $N$.
Set $w=1$ and replace $\la$ by $-\la$ in Theorem \ref{thm1}, 
then pick out the integrand which contains $e^{-i \si + \rho_\fa}$ and $e^\la$
to get
\begin{equation}
({\rm Ker}\,\bbD_M)_{\lambda + i\sigma}
 \cong \sum_{\chi \in \widehat{Z_M(M^0)}_\lambda}
 {\ind}_{M^\dagger AN}^G (\caH_{\eta_{\chi,-\la}} \otimes e^{-i \si + \rho_\fa} \otimes 1)
\otimes \caH_{(\bar{\chi} \otimes e^\la)} .
\label{adzz}
\end{equation}
The $\rho_\fa$-shift in (\ref{adzz}) is due to the definition
of unitarily induced representation; see for example \cite[VII \S1]{K1986}.
The theorem follows from (\ref{hap}) and (\ref{adzz}).$\hfill$$\Box$

We take this opportunity to revise an error in \cite{CH2016}.
For $G$ connected and with compact Cartan subgroup, \cite[Thm.B]{CH2016}
has a false expression
\[ \caH_{(X, \om)} =
\left ( {\sum}_{\la \in \im(\Phi) , \pi_\la \in \wh{G}_{disc}} 
\caH_{\eta_\la} \right )
\otimes \left ({\sum}_{w \in W_K} \bc_{-w \la}\right ) ,\]
as the summation over $W_K$ should not appear.

\begin{corollary} {\rm (Erratum to \cite[Thm.B]{CH2016}) }
If $G$ is connected and has a compact Cartan subgroup, then
\[ \caH_{(X, \om)} \cong
{\sum}_{\la \in {\rm Im}(i\Phi) , e^\la \in \wh{\tT}}
\caH_{\eta_{-\la}} \otimes \bbC_\la . \]
\label{spc}
\end{corollary}
\begin{proof}
Suppose that $G$ has a compact Cartan subgroup.
Then $A=1$, so the $\la + i \si$ of Theorem \ref{thm2} becomes
$\la$, and $\chi$ does not occur because $G$ is connected, so
${\ind}_P^G (\caH_{\eta_{\chi,-\la}} \otimes e^{-i \si + \rho_\fa} \otimes 1)
\otimes \caH_{(\bar{\chi} \otimes e^\la)}$
becomes $\caH_{\eta_{-\la}} \otimes \bc_\la$.
\end{proof}

\begin{remark} \label{rkm}
{\it Models of tempered representations.}
{\rm
According to Gelfand, a {\em model} of a compact Lie group is a
unitary representation which contains every equivalence class
of irreducible representation once \cite{GZ1984}.
For non-compact reductive Lie groups, this notion extends
to models of discrete series \cite{C2000} and principal series \cite{C2008}.
We briefly sketch the construction of models of standard tempered 
representations in \cite[\S4]{C2014}, taking advantage of the fact that
we make similar construction in Theorem \ref{thm2}.

There exist $(G \times \hH)$--invariant symplectic forms $\om_1, ..., \om_m$ 
on $X$ with moment maps $\Phi_1, ..., \Phi_m : X \lra i \fh^*$, such that 
the image of
$$ 
\begin{aligned} 
\bigcup_{j=1}^m &\{ \la + i \si \in \im(\Phi_j) \mid 
	e^{\la + i \si} \in \wh{\hH}\}/W_G  
	\lra \{\mbox{sum of tempered representations}\} \\
& \text{ given by } \la + i \si \mapsto
	{\sum}_{\chi \in \widehat{Z_M(M^0)}_\lambda}
	{\ind}_P^G (\caH_{\eta_{\chi,-\la}} \otimes 
		e^{-i \si + \rho_\fa} \otimes 1)
\end{aligned} 
$$
is multiplicity-free and contains every
equivalence class of standard tempered representations
(thus almost every tempered representation -- 
the missing tempered representations  have Plancherel measure zero).
By Theorem \ref{thm2},
$\sum_1^m \caH_{(X, \om_j)}$ is a model of tempered representations
in the sense that
every standard tempered representation occurs once.
}
\end{remark}

\section{Symplectic reduction}\label{sec6}
\setcounter{equation}{0}

Let $\om$ be a $(G \times H)$--invariant symplectic form on
$X = G \times \fhh$.  Let
$\mu \in \im(i\Phi) \subset i \fhh^*$ where $e^\mu \in \wh{\hH}$.
In this section, we carry out symplectic reduction \cite{MW1974}
for the right action of $\hH$
to obtain the symplectic quotient $(X_\mu,\om_\mu)$.
Then we apply geometric quantization to $(X_\mu,\om_\mu)$
and prove Theorem \ref{thm3}.

Recall symplectic reduction from \cite[\S5]{C2014}.
The moment map (\ref{cano}) of the right action of $\hH$ is
$\Phi : X \lra \fhh^*$.
There is a unique $v \in \fhh$ such that
$(i\Phi)^{-1}(\mu) = G \times \{v\} \subset X$.
Let $\imath$ and $\jmath$ be the maps in (\ref{ijij}).
Then there is a unique $G$--invariant symplectic form $\om_\mu$
on $G/\hH$ such that $\jmath^* \om_\mu = \imath^* \om$.
We have
\begin{equation}
 \om_\mu = d \mu \in \Om^2(G/\hH)^G .
 \label{ggg}
 \end{equation}
As $\mu \in \fhh^*$,
$d \mu \in \we^2 \fhh^* \subset \we^2 \fgg^* \cong \Om^2(G)^G$.
Furthermore $d \mu$ lies in the image of the natural injection
$\Om^2(G/\hH)^G \hookrightarrow \Om^2(G)^G$,
which explains (\ref{ggg}). The notation $d\mu$
does not imply that $\om_\mu$ is exact
(for example if $G/\hH$ is compact, 
it cannot have an exact symplectic form)
because $\mu$ does not lie in the image
of $\Om^1(G/\hH)^G \hookrightarrow \Om^1(G)^G$.
We obtain the symplectic quotient
\[ (X_\mu, \om_\mu) = (G/\hH, d \mu) .\]

We shall incorporate Dirac cohomology into the
geometric quantization of $(X_\mu, \om_\mu)$, so we modify the
line bundle for the spinor bundle over $G/\hH N$,
\[ B_\mu = G \times_\mu S \text{ defined by }
[ghn,s] = [g,\chi_\mu^{-1}(h)s] \in B_\mu \]
for all $g \in G$, $hn \in \hH N$ and $s \in S$.
Here $S$ is the same spinor as (\ref{gme}).
A section $f$ of $B_\mu$ can be identified with a function
$\psi : G/N \lra S$ such that
$\psi(gh) = e^\mu(h) \psi(g)$ for all $h \in H$,
given by $f(g) = [g, \psi(g)]$.
This gives a Hermitian structure on the sections by
$(f,f) = (\psi,\psi)$.

Recall that $G =KMAN$.
Here $G/\hH N$ has no $G$--invariant measure
because $\hH N$ is not unimodular, nevertheless
$G/\hH N$ has a measure $d(g H^0 N)$ which is $K$ and $M$-invariant
(\cite[Prop.8.44]{K2002}, \cite[p.2748]{C2014}). Let
\[ L^2(B_\mu) = \left \{\mbox{sections $f$ of $B_\mu$} \left |
\int_{G/\hH N} (f,f) \, d(gH^0 N) < \infty\right . \right \} .\]
The above correspondence $f \mapsto \psi$ leads to a $G$-equivariant map
\beq \label{indi}
L^2(B_\mu) \cong (L^2(G/N) \otimes S)_\mu .
\eeq

In (\ref{indi}), $(L^2(G/N) \otimes S)_\mu$
is the $\mu$-component of the direct integral decomposition of
$L^2(G/N) \otimes S$.  It is not a subspace, but is only weakly contained there.
In (\ref{dira0}), $\bd_M$ stabilizes each
$(L^2(G/N) \otimes S)_\mu$,
and we let $\bd_{M,\mu}$ denote the resulting operator.
It induces an operator $D_\mu$ on
$L^2(B_\mu)$, such that (\ref{indi})
intertwines $D_\mu$ with $\bd_{M,\mu}$.
Hence
\beq
{\rm Ker}\, D_\mu \cong {\rm Ker}\, \bd_{M,\mu}.
\label{tong}
\eeq

We define the quantization on the symplectic quotient as
\beq \caH_{(X_\mu, \om_\mu)} = {\rm Ker}\, D_\mu .
\label{yamm}
\eeq

\sk

\noindent {\it Proof of Theorem \ref{thm3}:}

Let $\mu = \la + i \si$ belong to the image $\im(i\Phi)$, 
where $e^\mu \in \wh{\hH}$.  Then
\[
\begin{array}{rll}
\caH_{(X,\om),\mu}
& \cong \sum_{\chi \in \widehat{Z_M(M^0)}_\lambda}
{\ind}_P^G (\caH_{\eta_{\chi,-\la}} \otimes e^{-i \si + \rho_\fa} \otimes 1)
& \mbox{by Theorem \ref{thm2}} \\
& \cong {\rm Ker}\, \bd_{M,\mu}
& \mbox{by Theorem \ref{thm1}} \\
& \cong {\rm Ker}\, D_\mu
& \mbox{by (\ref{tong})} \\
& = \caH_{(X_\mu, \om_\mu)}.
& \mbox{by (\ref{yamm})}
\end{array}
\]
This proves the theorem.$\hfill$$\Box$

\section{Background for general real reductive groups}\label{sec7}
\setcounter{equation}{0}
In this section we extend Theorems \ref{thm1}, \ref{thm2} and \ref{thm3}
from groups of Harish-Chandra class to general real reductive Lie groups.
Recall that the latter class, introduced in \cite{W1974a}, is given by
\begin{equation} \label{gen-real-class2}
\begin{aligned}
  &\text{(a) the Lie algebra $\fgg$ of $G$ is reductive}, \\
  &\text{(b) if $g \in G$ then $\Ad(g)$ is an inner automorphism of
        $\fgc$, and}\\
  &\text{(c) $G$ has a closed normal abelian subgroup $Z$ such that}\\
      &\phantom{XXii}  (i)\,\, Z \text{ centralizes } G^0\,,
                \text{ i.e. } Z \subset Z_G(G^0) \\
      &\phantom{XXi}  (ii)\,\, |G/ZG^0| < \infty \text{ and } \\
      &\phantom{XX}  (iii)\,\, Z \cap G^0 \text{ is co-compact in } Z_{G^0}\,.
\end{aligned}
\end{equation}
Without loss of generality we always assume $Z_{G^0} \subset Z$, so
(iii) becomes $Z \cap G^0 = Z_{G^0}$\,.

As $Z$ centralizes $G^0$ it centralizes $\fg$.  Thus $Z$ centralizes 
every Cartan subalgebra of $\fg$ and so it is contained in every 
Cartan subgroup of $G$.  The point is that, by definition, ``Cartan subgroup''
means the centralizer of a Cartan subalgebra.
We have to be careful here because it can happen 
that two $G$--conjugate Cartan subgroups of $G$ may fail to be $G^0$--conjugate.

Given a unitary character $\zeta \in \widehat{Z}$, define
\begin{equation}\label{def-gzeta}
\widehat{G}_\zeta = \{\pi \in \widehat{G} \mid 
	\pi(gz) = \zeta(z)\pi(g) \text{ for } g \in G \text{ and } z \in Z\}
\end{equation}
and
\begin{equation}
\label{def-l2zeta}	
L^2(G/Z;\zeta) = \left \{ f:G \to \bc \left |
	\begin{array}{l}  f \text{ is measurable} \\
		f(gz) = \zeta(z)^{-1}f(g)\, a.e.\, z \in Z, g \in G \\
		\int_{G/Z} |f(gZ)|^2 d(gZ) < \infty
	\end{array} \right . \right \}.	
\end{equation}
Induction by stages gives $L^2(G/Z;\zeta) = {\ind}_Z^G \bc_\zeta$ and
$L^2(G) = \int_{\zeta \in \widehat{Z}} L^2(G/Z;\zeta)d\zeta$.
But this is a bit redundant for $\widehat{G}$.
If $g \in G$ and $\zeta' = \Ad^*(g)\zeta$ then $\Ad^*(g): \widehat{G}_\zeta
\to \widehat{G}_{\zeta'}$ is a bijection (and homeomorphism for the 
hull--kernel topology) from $\widehat{G}_\zeta$ onto $\widehat{G}_{\zeta'}$,
sending $\pi' \in \widehat{G}_{\zeta'}$ to an equivalent representation
$\pi \in \widehat{G}_\zeta$.  This depends only on $gZG^0 \in G/ZG^0$.
Thus
$$
\Ad^*(g): L^2(G/Z;\zeta') \cong L^2(G/Z;\zeta) \text{ and }
\widehat{G} = {\bigcup}_{\zeta \in (\widehat{Z}/\Ad^*(G/ZG^0))} 
	\widehat{G_\zeta}.
$$
Here note that $\widehat{Z}/\Ad^*(G/ZG^0)$ is finite by
(\ref{gen-real-class2})(ii).

If $Z$ is noncompact the coefficients of a representation 
$\pi \in \widehat{G}_\zeta$ cannot be square integrable over $G$.
So we consider square integrability over $G/Z$.  This is
well defined because $G/ZG^0$ is finite.
More precisely, consider a coefficient 
$f_{\pi,u,v}(g) = \langle u, \pi(g)v\rangle$, $u,v \in \ch_\pi$\,.
Then $f_{\pi,u,v}(gz) = \zeta(z)^{-1}f_{\pi,u,v}(g)$, so $|f_{\pi,u,v}(gz)|
= |f_{\pi,u,v}(g)|$ for $g \in G$ and $z \in Z$, and $|f_{\pi,u,v}|$
is defined on $G/Z$.  We say that $f_{\pi,u,v}$ is {\em square
integrable} or {\em square integrable modulo $Z$} if
$f_{\pi,u,v} \in L^2(G/Z)$.  The following are equivalent
for $\pi \in \widehat{G}_\zeta$\,.
$$
\begin{array}{l}
	\text{(a) there exist nonzero } u,v \in \ch_\pi \text{ with } 
		f_{\pi,u,v} \in L^2(G/Z) \\
	\text{(b) } f_{\pi,u,v} \in L^2(G/Z) \text{ for every } u, v \in L^2(G/Z) \\
	\text{(c) } \pi \text{ is a (discrete) summand of the left 
	regular representation of } G \text{ on } L^2(G/Z;\zeta)
\end{array}
$$
Then we say that $\pi$ is a {\em relative discrete series representation }
of $G$.  The relative discrete series representations in $\widehat{G}_\zeta$
form the subset denoted $\widehat{G}_{\zeta,disc}$\,.
  
From Harish--Chandra's famous result, $G$ has relative discrete series
representations if and only if $G/Z$ has a compact Cartan subgroup.
When $G$ satisfies (\ref{gen-real-class2}), Levi components of cuspidal
parabolic subgroups also satisfy (\ref{gen-real-class2}).  Thus we can
construct the various tempered series more or less in the same way
as when $G$ is of Harish--Chandra class.  Further, the Plancherel
formula depends only on these tempered series.  See \cite{W1974a},
or \cite{W2018} for an update, or \cite{HW1986a} and \cite{HW1986b} for
a short direct proof.

Let $H$ be a Cartan subgroup of $G$.  As for Harish--Chandra class it is 
stable under a Cartan involution $\theta$, leading to a decomposition 
$H = T \times A$ and cuspidal parabolic subgroups $MAN$.  Here $T$ is a
Cartan subgroup of $M$, $Z \subset T$, and $T/Z$ is compact.  Thus,
given $\zeta \in \widehat{Z}$ we have the part $\widehat{M}_{\zeta,disc}$
of the relative discrete series of $M$ that transforms by $\zeta$, as
follows.  Let $\lambda \in i\ft^*$ such that (i) $e^\lambda$
is well defined on $T^0$, (ii) $\lambda$ is regular for $\Sigma_{\fmc,\ftc}$\,,
and (iii) $\zeta|_{Z\cap T^0} = e^\lambda_{Z\cap T^0}$\,.  Then as in
(\ref{eta0}) we have
\begin{equation}\label{eta0zeta}
\eta^0_\lambda \in \wh{M^0}_{\zeta,disc} \text{ with } 
        \Theta_{\eta^0_\lambda} 
        = \pm \tfrac{1}{\Delta_{M^0}} {\sum}_{w \in W_{M^0}} 
        \det(w)e^{w(\lambda)}
\text{ on } (M^0)''\cap T^0\,,
\end{equation}
as in (\ref{etadag}), where we avoid clutter by writing $\zeta$ instead
of $\zeta|_{(Z\cap M^0)}$\,.  Now we have
\begin{equation}\label{etadagzeta}
\eta^\dagger_{\chi,\lambda} :=
\chi \otimes \eta^0_\lambda \in \wh{M^\dagger}_{\zeta,disc} \text{ with }
\Theta_{\eta^\dagger_{\chi,\lambda}}(zm) = \tr \chi(z)
        \Theta_{\eta^0_\lambda}(m),
\end{equation}
and as in (\ref{def-ds-M}) we have $\eta_{\chi,\lambda} \in
\widehat{M}_{\zeta,disc}$ given by
\begin{equation}\label{def-ds-Mzeta}
\eta_{\chi,\lambda} := {\ind}_{M^\dagger}^M(\eta^\dagger_{\chi,\lambda})
\text{ with }
\Theta_{\eta_{\chi,\lambda}}(zm) = {\sum}_{xM^\dagger \in M/M^\dagger}
        \tr \chi(x^{-1}zx) \Theta_{\eta^0_\lambda}(x^{-1}mx).
\end{equation}
As before, $\eta_{\chi,\lambda}$ has infinitesimal character of Harish--Chandra
parameter $\lambda$ and formal degree 
$\deg(\eta_{\chi,\lambda}) = |M/M^\dagger|\deg(\chi)|\varpi(\lambda)|$,
and $\eta_{\chi,\lambda} \cong \eta_{\chi',\lambda'}$ just when their
$M^\dagger$--restrictions are equivalent.  Every representation in
$\widehat{M}_{\zeta,disc}$ is one of the $\eta_{\chi,\lambda}$ just described,
and the relative discrete series $\widehat{M}_{disc} 
= \bigcup_{\zeta \in \widehat{Z}} \widehat{M}_{\zeta,disc}$\,.

Similarly, if $\sigma \in \fa^*$, so $e^{i\sigma} \in \widehat{A}$,
\begin{equation}\label{def-ds-Gzeta}
\pi_{\chi,\lambda,\sigma}:= 
	{\ind}_{MAN}^G(\eta_{\chi,\lambda} \otimes e^{i\sigma} \otimes 1)
	= {\ind}_{M^\dagger AN}^G(\eta^\dagger_{\chi,\lambda} 
		\otimes e^{i\sigma} \otimes 1)
\end{equation}
is a unitary representation of $G$, and these representations ($\zeta$ fixed,
$\chi$ and $\lambda$ variable) form the $H$--series part of 
$\widehat{G}_\zeta$\,.  That $H$--series part depends only on the
$G$--conjugacy class of $(H,\zeta)$, and as these vary we sweep out all
but a set of measure zero in $\widehat{G}$.

As usual we fix a Cartan involution $\theta$ on $G$, a splitting
$\fg = \fk + \fp$ into $(\pm 1)$--eigenspaces of $d\theta$, and a 
nondegenerate $\Ad(G)$--invariant symmetric bilinear
form $b$ on $\fg$ that is negative definite on $\fk$, positive definite
on $\fp$ and satisfies $b(\fk,\fp) = 0$.  Our $\theta$--stable Cartan
subalgebra $\fh = \ft + \fa$ and we have the corresponding cuspidal
parabolic subalgebra $\fm + \fa + \fn$ of $\fg$, with an orthogonal 
direct sum decomposition $\fm = \ft + \fs$.  Using the spin module $S$
for the Clifford algebra $C(\fsc)$ the modified Dirac operator
$D_M = D_{(\fmc, \fkc \cap \fmc)} + i D_{diag;(\fkc \cap \fmc, \ftc)}$
of (\ref{mdirac}) is defined as in \cite[(1.1)]{CH2016} and its closure
$\bd_M$ as in (\ref{dira0}).  
This is possible because $D_M$ is defined in terms of the
Lie algebra, so its construction is the same as the construction for
groups of Harish--Chandra class.  

To be precise, note that $Z$ centralizes $N$ because $N \subset G^0$
and $Z$ centralizes $G^0$.  Further $Z$ centralizes $\fs$ and thus also 
the spin module $S$.  Now 
\begin{equation}\label{thm1za}
L^2(G/N) = \int_{\zeta \in \widehat{Z}} L^2(G/NZ;\zeta)d\zeta
\text{ and } L^2(G/N) \otimes S = 
\int_{\zeta \in \widehat{Z}} \left \{ L^2(G/NZ;\zeta) \otimes S\right \} d\zeta
\end{equation}
as unitary $G$--module, and
\begin{equation}\label{thm1zb}
\bd_M = \int_{\zeta \in \widehat{Z}}\bd_{\zeta,M} d\zeta \text{ where }
 \bd_{\zeta,M}: L^2(G/NZ;\zeta)\otimes S \to L^2(G/NZ;\zeta)\otimes S.
\end{equation}

Using $T = Z_M(M^0)$ every irreducible unitary representation of $H$ has form 
$\chi_{\lambda,\sigma}:=\chi \otimes e^\lambda \otimes e^{i\sigma}$
as in (\ref{mmm}).  In particular each such $\chi_{\lambda,\sigma}$
is finite dimensional.

The extension of Theorem \ref{thm1} to general real reductive Lie groups 
has both a relative formulation and an absolute
formulation.  The relative formulation is
\begin{theorem}\label{thm1rel}
Let $G$ be a general real reductive Lie group as in 
{\rm (\ref{gen-real-class})}.  Let $\zeta \in \widehat{Z}$.
In the notation of {\rm (\ref{thm1za})} and {\rm (\ref{thm1zb})},
$$
{\rm Ker}(\bd_{\zeta,M}) = 
	{\sum}_{\eta_{\chi, \lambda} \in\widehat M_{\zeta, disc}}
        \left (\int_{\sigma\in\fa^*} \ch_{\pi_{\chi,\lambda,\sigma}}
                \, d\sigma \right )
\otimes \left ({\sum}_{w\in W_{K \cap M}}
        \ch_{(\bar\chi\cdot w^{-1} \otimes e^{-w\lambda})} \right )
$$
and 
the natural action of $G \times T$ on ${\rm Ker}(\bd_{\zeta,M})$ is
$$
{\sum}_{\eta_{\chi, \lambda} \in\widehat M_{\zeta, disc}}
        \left (\int_{\sigma\in\fa^*} \pi_{\chi,\lambda,\sigma}
                \, d\sigma \right )
\otimes \left ({\sum}_{w\in W_{K \cap M}}
       (\bar\chi\cdot w^{-1} \otimes e^{-w\lambda}) \right ) .
$$
\end{theorem}
The absolute formulation of our extension of Theorem \ref{thm1} is
\begin{corollary}\label{thm1abs}
In the notation of {\rm (\ref{thm1za})} and {\rm (\ref{thm1zb})},
$$
{\rm Ker}(\bd_M) = \int_{\zeta \in \widehat{Z}}
	\left \{ {\sum}_{\eta_{\chi, \lambda} \in\widehat M_{\zeta, disc}}
        \left (\int_{\sigma\in\fa^*} \ch_{\pi_{\chi,\lambda,\sigma}}
                \, d\sigma \right )
\otimes \left ({\sum}_{w\in W_{K \cap M}}
    \ch_{(\bar\chi\cdot w^{-1} \otimes e^{-w\lambda})} \right )\right \} d\zeta
$$
and the natural action of $G \times T$ on ${\rm Ker}(\bd_M)$ is
$$
\int_{\zeta \in \widehat{Z}}
        \left \{ {\sum}_{\eta_{\chi, \lambda} \in\widehat M_{\zeta, disc}}
        \left (\int_{\sigma\in\fa^*} \pi_{\chi,\lambda,\sigma}
                \, d\sigma \right )
\otimes \left ({\sum}_{w\in W_{K \cap M}}
	(\bar\chi\cdot w^{-1} \otimes e^{-w\lambda}) \right ) \right \}d\zeta.
	$$
\end{corollary}

Corollary \ref{thm1abs} is an immediate consequence of Theorem \ref{thm1rel};
one just integrates over $\widehat{Z}$ modulo the conjugation action of
$G/ZG^0$.
We will prove Theorem \ref{thm1rel} in Section \ref{sec9}.  This will
use an extension, in Section \ref{sec8}, of a reduction method developed 
in \cite{W1974a}.

The discussion leading to the statement of Theorem \ref{thm2}, 
and also Theorem \ref{assam}, is valid for general real
reductive Lie groups with essentially no modification.  We have 
$(G \times H)$--invariant symplectic forms $\omega$ on $X = G \times \fh$.
For each such $\omega$ we have a complex line bundle $\bl \to X$ with
$c_1(\bl) = \omega$.  The inclusion $\imath : G \to X$, $\imath(g) = (g,0)$, 
pulls $\bl$ back to a line bundle $\imath^*\bl \to G$.  Fix 
$\zeta \in \widehat{Z}$. Then we have
$$
\begin{aligned}
& L^2(\bl)^N_\zeta \text{: measurable sections $f$ of $\bl \to X$ such that} \\
&\phantom{XXXX}\text{$f(gnz,x) = \zeta(z)^{-1}f(g,x)$ and 
	$\int_{G/NZ \times \fh} |f(g,x)|^2 d(gNZ)dx < \infty$, and} \\
& L^2(\imath^*\bl)^N_\zeta \text{: measurable sections $f$ of $\imath^*\bl \to G$ 
	such that} \\
&\phantom{XXXX}\text{$f(gnz) = \zeta(z)^{-1}f(g)$ and 
	$\int_{G/NZ} |f(g)|^2 d(gNZ) < \infty$.}
\end{aligned}
$$
Let $\bd_{\zeta,\bl}$ denote the Dirac operator on 
$L^2(\imath^*\bl)^N_\zeta \otimes S$.  Then we have a unitary 
representation $\pi_{\zeta,(X,\omega)}$ with representation space
$$
\ch_{\zeta,(X,\omega)} = \{f \in L^2(\bl)^N_\zeta \mid
	\imath^*f \in \text{Ker\,}\bd_{\zeta,\bl} \text{ and $f$ is
	$\pi$--holomorphic}\}.
$$
Now recall that the action of $G\times  H$ on $(X,\omega)$ is
Hamiltonian, so the right action of $H$ has moment map
$\Phi: X \to \fh^*$ as in the case where $G$ is of Harish--Chandra class.
The relative version of the extension of Theorem \ref{thm2} is
\begin{theorem}\label{thm2rel}
Let $G$ be a general real reductive Lie group as in 
{\rm (\ref{gen-real-class})}.
Fix $\zeta \in \widehat{Z}$.  The unitary representation space
$\ch_{\zeta,(X,\omega)}$ for $G \times T$ is
$$
\int_{\tiny \begin{array}{l} \la + i\sigma \in \im(i \Phi)\\
	\exp(\la + i\sigma) \in \wh{H^0_\zeta} \end{array}}
{\sum}_{\tiny \chi \in \widehat{Z_M(M^0)}_\lambda}
  \left ( {\ind}_P^G(\caH_{\bar\chi \otimes \eta^0_{-\la}}
       \otimes e^{-i\sigma +\rho_\fa} \otimes 1)
       \otimes \ch_{(\chi\otimes e^\lambda)} \right ) \, d \la \, d \si
$$
where $\widehat{Z_M(M^0)}_\lambda$ denotes the elements of $Z_M(M^0)$
that agree with $e^\lambda$ on $Z_{M^0}$, and thus agree with $\zeta$
on $Z$.
\end{theorem}

Now we sum over $\widehat{Z}$.  Let $\bd_\bl$ denote the Dirac operator on 
$L^2(\imath^*\bl)^N \otimes S$ and $\pi_{(X,\omega)}$ the corresponding
representation on
$$
\ch_{(X,\omega)} = \{f \in L^2(\bl)^N \mid
        \imath^*f \in \text{Ker\,}\bd_\bl \text{ and $f$ is
        $\pi$--holomorphic}\}.
$$
The absolute version of the extension of Theorem \ref{thm2} is
\begin{corollary}\label{thm2abs}
The unitary representation space
$\ch_{(X,\omega)}$ for $G \times T$ is
$$
\int_{\tiny \begin{array}{l} \la + i\sigma \in \im(i \Phi)\\
        \exp(\la + i\sigma) \in \wh{H^0} \end{array}}
{\sum}_{\tiny \chi \in \widehat{Z_M(M^0)}_\lambda}
  \left ( {\ind}_P^G(\caH_{\bar\chi \otimes \eta^0_{-\la}}
       \otimes e^{-i\sigma +\rho_\fa} \otimes 1)
       \otimes \ch_{(\chi\otimes e^\lambda)} \right ) \, d \la \, d \si
$$
where $\widehat{Z_M(M^0)}_\lambda$ denotes the elements of $Z_M(M^0)$
that agree with $e^\lambda$ on $Z_{M^0}$.
\end{corollary}

Again, we will prove Theorem \ref{thm2rel} in Section \ref{sec9},  using
an extension, from Section \ref{sec8}, of a reduction method developed 
in \cite{W1974a}.

Theorem \ref{thm3}, the principle that quantization commutes with reduction,
is valid as stated for general real reductive Lie groups.  We will go over
the argument toward the end of Section \ref{sec9}.

\section{Reduction to the case of compact center}\label{sec8}
\setcounter{equation}{0}

In this section we state and prove Theorem \ref{Gdchi}, which will 
reduce the proofs of Theorems \ref{thm1rel} and
\ref{thm2rel} to the case where $G^\dagger = Z_G(G^0)G^0$ has compact center,
so that we can identify discrete series representations by lowest $K$--type.

Since $Z$ centralizes $G^0$ we have $Z_{G^0} \subset Z \subset Z_G(G^0)$.
Let $\zeta \in \widehat{Z}$\,, $\chi \in \widehat{Z_G(G^0)}_\zeta$\,, 
$\ch_\chi$ its representation
space, and $U = U(\ch_\chi)$ the unitary group of $\ch_\chi$.
Recall that $\dim\caH_\chi < \infty$, so $U$ is compact.  Of course we
have the defining representation $1_U \in \widehat{U}$.  It is the usual
representation of $U(\ch_\chi)$ on $\ch_\chi$\,, given by $1_U(z) = z$.
If $L$ is any closed subgroup of $G$ of the form $Z_G(G^0)L^0$ then we
denote
\begin{equation}\label{ell}
L[\chi] = (U \times L)/\{(\chi(z)^{-1},z) \mid z \in Z_G(G^0)\}.
\end{equation}
In particular we have the quotient groups
\begin{equation}\label{defGdchi}
\begin{aligned}
G^\dagger[\chi] = &(U \times G^\dagger)/\{(\chi(z)^{-1},z) 
	\mid z \in Z_G(G^0)\},\\
G[\chi] = &(U \times G)/\{(\chi(z)^{-1},z) 
        \mid z \in Z_G(G^0)\}.
\end{aligned}
\end{equation}
Note that $G^\dagger[\chi]$ is the identity component of $G[\chi]$.
They are general real reductive Lie groups as in (\ref{gen-real-class}).
We write $p$ for the restriction to $G$  of the projection
$(U \times G) \to G[\chi]$ and also for the restriction to 
$G^\dagger$ of $(U \times G^\dagger) \to G^\dagger[\chi]$.
Then $p$ induces an isomorphism 
$G/G^\dagger \cong G[\chi]/G^\dagger[\chi]$.

\begin{lemma}\label{gchi}
$G^\dagger[\chi]$ is a connected reductive Lie group with Lie algebra
$\fu \oplus (\fg^\dagger/\fz)$.  It has compact center $U$.  
For appropriate normalizations
of Haar measures, $f \mapsto f\cdot p$ defines an equivariant isometry
of $L^2(G^\dagger[\chi]/U,1_U)$ onto $L^2(G^\dagger/Z_G(G^0),\chi)$.
\end{lemma}

\begin{proof}  We follow the proof of \cite[Lemma 3.3.2]{W1974a}, which is
the case $G^\dagger = ZG^0$.  There $\chi = \zeta$ and $U$ is the circle 
group $U(1)$.  If $f \in L^2(G^\dagger[\chi]/U,1_U)$, $z \in Z_G(G^0)$ and 
$g \in G^\dagger$ then
$$
(f\cdot p)(gz) = f(1,gz) = f(\chi(z),g) = \chi(z)^{-1}f(1,g)
	= \chi(z)^{-1}(f\cdot p)(g)
$$
and
$$
\begin{aligned}
\int_{G^\dagger/Z_G(G^0)} |(f\cdot p)(g)|^2 d(gZ_G(G^0))
	&= \int_{(U\times G^\dagger)/(U\times Z_G(G^0))} |f(u,g)|^2
		d(uU \times gZ_G(G^0)) \\
	&= \int_{G^\dagger[\chi]/U} |f(\overline{g})|^2 d(\overline{g}U).
\end{aligned}
$$
Thus $f \mapsto f\cdot p$ is an isometric injection of 
$L^2(G^\dagger[\chi]/U,1_U)$ into $L^2(G^\dagger/Z_G(G^0),\chi)$.
It is surjective because any $f' \in L^2(G^\dagger/Z_G(G^0),\chi)$
has inverse image $f(z,g) = \chi(z)^{-1}f'(g)$.
\end{proof}

Let $P = MAN$ be a cuspidal parabolic subgroup of $G$ associated to a 
Cartan subgroup $H = T \times A$.  Recall some
properties of $M^\dagger = Z_M(M^0)M^0$.  First, 
$\widehat{M^\dagger}$ consists of the $\varphi \otimes \eta^0$ where
$\varphi \in \widehat{Z_M(M^0)}$ agrees with $\eta^0 \in \widehat{M^0}$ 
on $Z_{M^0}$\,.  Here we write $\varphi$ instead of $\chi$ to avoid the 
possibility of confusion in Section \ref{sec9}, but to avoid cluttered 
notation we continue to use $U$ for $U(\ch_\varphi)$.  
$\widehat{M^\dagger}_\varphi$ denotes the subset
of $\widehat{M^\dagger}$ corresponding to a fixed $\varphi$.  
The relative discrete series of $M^\dagger$ consists of
the $\varphi \otimes \eta^0$ where $\eta^0 \in \widehat{M^0}_{disc}$\,, i.e.
the representations $\eta^\dagger_{\varphi,\lambda} = 
\varphi \otimes \eta^0_\lambda$ of $M^\dagger$ as in (\ref{etadagzeta}),  
and $\widehat{M^\dagger}_{\varphi,disc} = \widehat{M^\dagger}_\varphi \cap 
\widehat{M^\dagger}_{disc}$.  As in (\ref{defGdchi}) we have
\begin{equation}\label{defMphi}
\begin{aligned}
M^\dagger[\varphi] = &(U \times M^\dagger)/\{(\varphi(z)^{-1},z) 
        \mid z \in Z_M(M^0)\},\\
M[\varphi] = &(U \times M)/\{(\varphi(z)^{-1},z) 
        \mid z \in Z_M(M^0)\}.
\end{aligned}
\end{equation}

The key observation of this section is this extension of
\cite[Theorem 3.3.3]{W1974a}.

\begin{theorem}\label{Gdchi}
The map $\varepsilon^\dagger_\chi: \widehat{G^\dagger[\chi]}_{1_U} \to 
\widehat{G^\dagger}_\chi$\,,
given by $\varepsilon^\dagger_\chi(\psi) = \psi\cdot p$,  
is a well defined bijection and maps
$\widehat{G^\dagger[\chi]}_{1_U,disc}$ onto 
$\widehat{G^\dagger}_{\chi,disc}$\,. It
carries Plancherel measure of $\widehat{G^\dagger[\chi]}_{1_U}$ to Plancherel
measure of $\widehat{G^\dagger}_\chi$\,. Distribution characters
satisfy $\Theta_{\varepsilon^\dagger_\chi(\psi)} = \Theta_\psi \cdot p$.

Similarly, $\varepsilon^\dagger_\varphi: \widehat{M^\dagger[\varphi]}_{1_U} \to 
\widehat{M^\dagger}_\varphi$\,,
given by $\varepsilon^\dagger_\varphi(\psi) = \psi\cdot p$,  
is a well defined bijection and maps
$\widehat{M^\dagger[\varphi]}_{1_U,disc}$ onto 
$\widehat{M^\dagger}_{\varphi,disc}$\,. It
carries Plancherel measure of $\widehat{M^\dagger[\varphi]}_{1_U}$ to Plancherel
measure of $\widehat{M^\dagger}_\varphi$\,. Distribution characters
satisfy $\Theta_{\varepsilon^\dagger_\varphi(\psi)} = \Theta_\psi \cdot p$.
\end{theorem}

Our argument for Lemma \ref{gchi} was a perturbation of the
proof of the special case \cite[Lemma 3.3.2]{W1974a}.  Similarly, 
the proof of \cite[Theorem 3.3.3]{W1974a}, which occupies most of
\cite[Section 3.3]{W1974a}, goes through with no serious change, 
yielding the argument for Theorem \ref{Gdchi}.  One need only
be careful about noncommutativity of $U$ when $\dim \ch_\chi > 1$
or $\dim \ch_\varphi > 1$.

\section{Proofs for general real reductive groups}\label{sec9}
\setcounter{equation}{0}

Recall that $K$ denotes the fixed point set of the Cartan involution
$\theta$ of $G$, so $Z_{G^0} \subset Z \subset Z_G(G^0) \subset K$.  
$K/Z$ is a maximal compact subgroup of $G/Z$ and $K/Z_G(G^0)$ is a
maximal compact subgroup of $G/Z_G(G^0)$.
Since $K$ is the normalizer of its Lie algebra $\fk$ it meets every
component of $G$.  Also, $K\cap G^0 = K^0$ and $K\cap G^\dagger 
= Z_G(G^0)K^0$.  We will write $K^\dagger$ for this group $Z_G(G^0)K^0$.
The construction (\ref{ell}) gives us
\begin{equation}\label{defKchi}
\begin{aligned}
K^\dagger[\chi] = &(U \times K^\dagger)/\{(\chi(z)^{-1},z) 
        \mid z \in Z_G(G^0)\},\\
K[\chi] = &(U \times K)/\{(\chi(z)^{-1},z) 
        \mid z \in Z_G(G^0)\}.
\end{aligned}
\end{equation}

\begin{lemma}\label{maxcpt}
$K^\dagger[\chi]$ is a maximal 
compact subgroup of $G^\dagger[\chi]$\,, $K[\chi]$ is a maximal
compact subgroup of $G[\chi]$, and $p$ induces isomorphisms $G/G^\dagger \cong
G[\chi]/G^\dagger[\chi] \cong K[\chi]/K^\dagger[\chi]$.
\end{lemma}

Denote $K_M = K\cap M$ and $K_M^\dagger = K \cap M^\dagger$, so
$K_M^\dagger = Z_M(M^0)K_M^0$\,.  Now
\begin{equation}\label{defKphi}
\begin{aligned}
K_M^\dagger[\varphi] = 
	&(U \times K_M^\dagger)/\{(\varphi(z)^{-1},z) 
        \mid z \in Z_M(M^0)\},\\
K_M[\varphi] = &(U \times K_M)/\{(\varphi(z)^{-1},z) 
        \mid z \in Z_M(M^0)\}.
\end{aligned}
\end{equation}

Applying the character formulas of (\ref{eta0zeta}) and (\ref{etadagzeta}),
and using the map $\varepsilon_\varphi^\dagger$ of Theorem \ref{Gdchi}, 
we have the following.

\begin{lemma}\label{resKzeta}
Let $\varphi \in \widehat{Z_M(M^0)}_\lambda$ such that 
$\eta^\dagger_{\varphi,\lambda} \in \widehat{M^\dagger}_{\varphi,disc}$.  
Let $\psi^\dagger_{\varphi,\lambda} \in \widehat{M^\dagger[\varphi]}_{1_U,disc}$
such that $\varepsilon_\varphi^\dagger(\psi^\dagger_{\varphi,\lambda}) 
= \eta^\dagger_{\varphi,\lambda}$
Then the restriction of the character of $\psi^\dagger_{\varphi,\lambda}$ from
$M^\dagger[\varphi]$ to $K_M^\dagger[\varphi]$ is equal 
to the character of the restriction of $\psi^\dagger_{\varphi,\lambda}$ from 
$M^\dagger[\varphi]$ to $K_M^\dagger[\varphi]$.  In 
other words the relative discrete series characters satisfy \hfill\newline
\centerline{$
\Theta_{\psi^\dagger_{\varphi,\lambda}}|_{K_M^\dagger[\varphi]} 
 = \Theta_{\psi^\dagger_{\varphi,\lambda}|_{K_M^\dagger[\varphi]}}$} 
and thus \hfill\newline
\centerline{$
\Theta_{\eta^\dagger_{\varphi,\lambda}}|_{K_M^\dagger}
 = \Theta_{\eta^\dagger_{\varphi,\lambda}|_{K_M^\dagger}}$}\,.
\end{lemma}

We use Theorem \ref{Gdchi} and the character formula in (\ref{def-ds-Mzeta}) 
to carry the result of Lemma \ref{resKzeta} from $M[\varphi]$ to $M$.

\begin{proposition}\label{resK}
Let $\varphi \in \widehat{Z_M(M^0)}_\lambda$ such that
$\eta_{\varphi,\lambda} \in \widehat{M}_{\varphi,disc}$.
Let $\psi_{\varphi,\lambda} \in \widehat{M[\varphi]}_{1_U,disc}$
such that $\varepsilon_\varphi\psi_{\varphi,\lambda} = 
	\eta_{\varphi,\lambda}$\,.
Then\hfill\newline
\centerline{$
\Theta_{\psi_{\varphi,\lambda}}|_{K_M[\varphi]}
 = \Theta_{\psi_{\varphi,\lambda}|_{K_M[\varphi]}}$}
and thus \hfill\newline
\centerline{$
\Theta_{\eta_{\varphi,\lambda}}|_{K_M}
 = \Theta_{\eta_{\varphi,\lambda}|_{K_M}}$}\,.
\end{proposition}
By $K_M$--type of $\eta_{\varphi,\lambda}$ we mean, as usual, an 
irreducible
summand $\tau$ of $\eta_{\varphi,\lambda}|_{K_M}$\,.  It has form 
$\tau = {\ind}_{K_M^\dagger}^{K_M}\,\tau^\dagger$ where $\tau^\dagger
= \varphi\otimes \tau^0$ with $\tau^0 \in \widehat{K_M^0}$. 
With respect to a positive root system $\tau^0$
has some highest weight $\nu$, and $\nu$ also is the highest weight of
$\tau^\dagger = \varphi \otimes \tau^0$.  The restriction of $\tau$ to
$K_M^\dagger$ is the sum over $M/M^\dagger$ of conjugates 
$\tau^\dagger\cdot \Ad(m)^{-1}$ of $\tau^\dagger$.  For brevity 
we will say that $\nu$ is the highest weight of $\tau$.

We use the character formulas of (\ref{eta0zeta}), 
(\ref{etadagzeta}) and (\ref{def-ds-Mzeta}), or we can rely on
\cite{V1984} or \cite{VZ1984}, for the following corollary.  It extends 
a case of Theorem \cite[Theorem 5.3]{HP2002}.

\begin{corollary}\label{low-k-type}
Let $\varphi \in \widehat{Z_M(M^0)}_\lambda$ such that
$\eta_{\varphi,\lambda} \in \widehat{M}_{\varphi,disc}\,,$
where we choose the 
positive root system so that $\langle \alpha,\lambda|_\ft \rangle \geqq 0$ 
for all $\fh$--roots of $\fn$.  Let $\rho_{nonc}$ denote half the sum
of the noncompact roots of $\fn$.  Then 
$\eta_{\varphi,\lambda}$ has lowest $K_M$--type of highest weight
$\lambda + 2\rho_{nonc}$\,.
\end{corollary}

\medskip
\noindent {\it Proof of Theorems \ref{thm1rel} and \ref{thm2rel}.}

The delicate point in the proofs of Theorems \ref{thm1} and \ref{thm2}
is their dependence on \cite[Proposition 5.4]{HP2002}.
The argument of \cite[Proposition 5.4]{HP2002} relies on \cite{V1984}
(or see \cite{VZ1984}) for the existence of a $K$--type of a 
certain highest weight $\lambda + \rho_{nonc}$ in Dirac
cohomology modules of groups of Harish--Chandra class.
Corollary \ref{low-k-type} provides the corresponding existence
result for the groups $M$. Now our arguments for Theorems \ref{thm1} 
and \ref{thm2} 
go through with only minor changes for $G$ and the representations 
in the $\widehat{G}_\zeta$\,.  Theorems
\ref{thm1rel} and \ref{thm2rel} follow.  $\hfill$$\Box$

\medskip
\noindent {\it Proof of Theorem \ref{thm3} for general real reductive Lie 
groups.}

The discussion in Section \ref{sec6} goes through with only trivial changes
for general real reductive Lie groups.  The point is that 
$Z \subset Z_G(G^0) \subset H$ because they centralize $\fh$, so we can replace 
$H^0$ by $ZH^0$ in the integration 
that defines $L^2(B_\mu)$.  Then we proceed relative to $\zeta \in \widehat{Z}$
as usual with $e^\mu = \zeta$ on $Z\cap H^0$ and $\lambda = \mu + i\sigma$.  
That gives us relative
versions of (\ref{tong}) and (\ref{yamm}), and now the proof goes as in
Section \ref{sec6}:
\[
\begin{array}{rll}
\caH_{(X,\om),\mu}
& \cong \sum_{\chi \in \widehat{Z_M(M^0)}_\lambda}
{\ind}_P^G (\caH_{\eta_{\chi,-\la}} \otimes e^{-i \si + \rho_\fa} \otimes 1)
& \mbox{by Theorem \ref{thm2rel}} \\
& \cong {\rm Ker}\, \bd_{M,\mu}
& \mbox{by Theorem \ref{thm1rel}} \\
& \cong {\rm Ker}\, D_\mu
& \mbox{by the extension of (\ref{tong})} \\
& = \caH_{(X_\mu, \om_\mu)}.
& \mbox{by the extension of (\ref{yamm})}
\end{array}
\]
As before, this proves the theorem.$\hfill$$\Box$

\end{document}